\documentclass[12pt]{amsart}
\usepackage[margin=0.8in]{geometry}
\usepackage{xspace}
\usepackage{amsfonts}
\usepackage{amsthm}
\usepackage{amssymb}
\usepackage{times} 
\usepackage{graphicx}
\usepackage{hyperref}
\usepackage{comment}
\usepackage{xcolor}
\usepackage[normalem]{ulem}
\usepackage[shortlabels]{enumitem}
\date{}
\usepackage[normalem]{ulem}
\usepackage{amsmath}

\newcommand{\stkout}[1]{\ifmmode\text{\sout{\ensuremath{#1}}}\else\sout{#1}\fi}
\numberwithin{equation}{section}
\newtheorem{thm}{Theorem}[section]
\newtheorem{prop}[thm]{Proposition}

\newtheorem{lemma}[thm]{Lemma}

\newtheorem{defn}[thm]{Definition}
\newtheorem{cor}[thm]{Corollary}
\theoremstyle{definition}
\newtheorem{remark}[thm]{Remark}
\newtheorem{example}[thm]{Example}

\title{Differential Codes on Higher Dimensional Varieties Via Grothendieck's Residue Symbol}

\author{David Grant}
\address{Department of Mathematics, University of Colorado Boulder, Boulder, CO 80309-0395, USA}

\author{John D. Massman, III}

\address{
Virginia Mason Franciscan Health, 1149 Market Street, Tacoma, WA 98402, USA}

\author{S. Srimathy}

\address{School of Mathematics, Tata Institute of Fundamental Research, Mumbai, India}


\begin{document}
\begin{abstract}
    We give a new construction of linear codes over finite fields on higher dimensional varieties using Grothendieck's theory of residues. This generalizes the construction of  differential codes over curves to varieties of higher dimensions.
\end{abstract}
\maketitle


\section{Introduction}\label{sec:intro}

To fix notation, let $k$ be a finite field 
and $X$ be a smooth projective variety\footnote{For us, a variety over $k$ is a geometrically integral separated scheme of finite type over $k$. We denote that $P$ is a scheme-theoretic point of $X$ by writing $P\in X$. For an extension $K/k$, we denote that $P$ is a $K$-rational point of $X$ by writing $P\in X(K)$. } of dimension $r$ over $k$. 
For a point $P \in X$, $\mathcal{O}_P$ denotes its local ring and $m_P$ denotes the maximal ideal of $\mathcal{O}_P$. 
Let $\Omega^r(X)$ denote the global sections of the sheaf of $k$-rational $r$-differential forms on $X$.
For a divisor $D$ of $X$, $Supp(D)$ denotes its support.  
The divisor associated to a rational function $f\in k(X)^*$ (respectively a  $k$-rational $r$-differential $0\neq \omega\in \Omega^r(X)$) is denoted by $(f)$ (respectively $(\omega)$).  For a $k$-rational divisor $D$ on $X$, 
let 
\begin{align*}
L(D) = \{ f \in k(X)^* | (f) + D \geq 0 \} \cup \{0\}
\end{align*}
denote the global sections of the associated invertible sheaf $\mathcal{O}(D)$ and let 

\begin{align*}
\Omega^r(D) = \{0\neq \omega \in \Omega^r(X) | (\omega) + D \geq 0\}\cup\{0\}
\end{align*}
denote the differentials in $\Omega^r(X)$ with poles bounded by $D$.

\indent When $X=C$ is a smooth projective curve over $k$ with a given set of $k$-rational points $\mathcal{P} = \{P_1, P_2, \dots, P_n\}$,  there are several ways to construct linear codes which we now briefly recall. More details can be found in \cite[Chapter 3.1]{tsfasman}. Let $G$  be a $k$-rational divisor  on $C$ with $ Supp(G) \cap \mathcal{P} = \emptyset$. Set $D= \sum_{i=1}^n P_i$. Goppa constructed the \emph{differential code} $C_{\Omega}(\mathcal{P}, G)$ that now bears his name as the image of the residue map 
\begin{align}\label{eqn:res_curves}
\begin{split}
    Res_{(\mathcal{P},G)} : \Omega^1(D - G) &\rightarrow k^n \\
                           \omega &\mapsto (Res_{P_1}\omega,  Res_{P_2}\omega, \dots, Res_{P_n}\omega )
\end{split}                           
\end{align}
where for a local parameter $t$ of $\mathcal{O}_{P_i}$ and $\omega = f dt, f \in k(C)$, $Res_{P_i} \omega$ is the coefficient of $t^{-1}$ in the Laurent series expansion 
\begin{align*}
    f = \sum_{j\in \mathbb{Z}} c_j t^j.
\end{align*}
Recall that $Res_{P_i} \omega$ is independent of the choice of the local parameter $t$ (\cite[Proposition 2.2.19]{tsfasman}).\\
The \emph{functional code} $C_L(\mathcal{P}, G)$ on $C$ is defined to be the image of the evaluation map 
\begin{align*}
    Ev_{(\mathcal{P}, G)}: L(G) &\rightarrow k^n \\
     f&\mapsto (f(P_1), f(P_2), \dots, f(P_n)).
\end{align*}

For the wonderful properties of these codes and their importance in coding theory we refer the reader to \cite{stitch}, \cite{walker}  and \cite{hoholdt}. We only mention  that (i) although functional codes get the most attention, differential codes are useful for decoding (\cite{decoding}); (ii)  $C_{\Omega}(\mathcal{P}, G)$ is the dual code to $C_L(\mathcal{P}, G)$, and (iii) as a further testament to their dual nature, one can show that every functional code is differential and vice-versa.

Let $\bar{k}$ be an algebraic closure of $k.$
Recall from \cite[Theorem 3.1.43]{tsfasman} that the proof of duality stated in  (ii) is almost a direct application of the \emph{residue theorem} on algebraic curves (\cite[Proposition 2.2.20]{tsfasman}),which states that for any $\omega \in \Omega^1_k(C)$.

\begin{align}\label{eqn:res_formula_curves}
    \sum_{P \in X(\overline{k})}Res_P \omega =0
\end{align}
together with the fact that  for any $f \in L(G)$,
\begin{equation}\label{linearity}
Res_{P_i}f\omega=f(P_i)Res_{P_i}\omega
\end{equation}
since $\omega\in\Omega^1(D-G)$ has at worst simple poles at the $P_i$. This fails in general if $\omega$ has a higher order pole.

The analogue of the functional construction for higher dimensional varieties followed almost immediately (\cite[Chapter 1, \S3]{lin_codes_modular_curves}, \cite[\S 3.1.1]{tsfasman}):

Let $X$ be an $r$-dimensional smooth projective variety over $k$. Given a set $\mathcal{P} = \{P_1, P_2 ,\dots , P_n\}$ of $k$-rational points on $X$ and a $k$-rational divisor $G$ with $Supp(G) \cap \mathcal{P} = \emptyset$, the \emph{functional code} $C_L(\mathcal{P}, G)$ is defined to be the image of the evaluation map 
\begin{align}\label{eqn:functional}
\begin{split}
    Ev_{(\mathcal{P}, G)}: L(G) &\rightarrow k^n \\
     f&\mapsto (f(P_1), f(P_2), \dots, f(P_n)).
     \end{split}
\end{align}

For a survey of work on these codes, we refer the reader to \cite{little}. Recall that the estimate of the minimum distance, which is easy in the case of codes from curves, becomes particularly elusive as soon as the variety has dimension greater than 1.

Much less work has been done on the higher dimensional analogue of the differential construction of codes, no doubt because of the difficulty of using higher dimensional residue formulas.  Getting a new description of codes involving differential forms may be helpful towards estimating the parameters of functional codes or their duals. The first approach to this was by the second-named author in his 2005 dissertation (heretofore unpublished), who used Grothendieck's theory of residues (\cite{massman}) as follows.   

For $X$ as above, let $\mathcal{D} = \{D_1,\dots,D_r\}$ be effective divisors on $X$ that intersect properly at a finite set of points (i.e, their set-theoretic intersection $\cap_iD_i$ is zero-dimensional, see (\cite[\href{https://stacks.math.columbia.edu/tag/0AZQ}{Tag 0AZQ}]{stacks-project}) for details). Taking a finite extension of $k$ if necessary, we can assume that all $D_i$ and all their points of intersection are rational over $k$. 
For $\omega\in \Omega^r(D)$ one can define   the notion of \emph{residue of $\omega$ relative to $\mathcal{D}$}, denoted by $Res_{P} \begin{bmatrix} \omega \\ {D_1, D_2, \dots, D_r} \end{bmatrix}$ for every point $P\in X$, that vanishes for $P\notin \cap_i{D_i}$, and satisfies the Grothendieck's Residue Theorem (see \S\ref{sec:groth_residues} for details) analogous to (\ref{eqn:res_formula_curves}): 

\begin{align*}
    \sum_{P \in X(\overline{k})} Res_{P} \begin{bmatrix} \omega  \\ {D_1, D_2, \dots, D_r} \end{bmatrix}  = 0. 
\end{align*}

With the above set up, Massman constructed a \emph{differential code} on $X$ denoted  $C_{\Omega}(\mathcal D, \mathcal{P}, G),$  where $\mathcal{P} = \cap_i D_i $  and $G$ is a divisor disjoint from $\mathcal{P}$ \cite{massman}.  He showed that
$C_{\Omega}(\mathcal D, \mathcal{P}, G)$ is contained in the dual code to $C_L(\mathcal{P}, G)$ if $D_1,\dots,D_r$ intersect transversally (i.e, the intersection multiplicity is $1$ at each point of intersection; see \cite[\href{https://stacks.math.columbia.edu/tag/0AZR}{Tag 0AZR}]{stacks-project} for more details).

In \cite{alain}, although aware that one could use Grothendieck's theory of residues, Couvreur did the community a great service by independently developing the theory of  differentials and residues on surfaces and using it to construct  differential codes that satisfy properties analogous to the case of curves (see Remark \ref{sec:couv} for a comparison of the theories). 
In the case that $D_1$ and $D_2$ satisfy certain conditions (so-called ``$\Delta$-convenience"), which includes the analogue of (\ref{linearity}), he was able to show that $C_{\Omega}(\mathcal D, \mathcal{P}, G)$ is contained in the dual of $C_L(\mathcal{P}, G),$ that every functional code is differential and that every differential code is functional. He also gave examples showing that, unlike the case for curves, a differential code associated to $(\mathcal{P}, G)$ could be a proper subspace of the dual of the functional code.

For more recent work on codes from higher dimensional varieties, see for example, \cite{alain2}, \cite{alain3}, \cite{newer1}, \cite{newer2}, and \cite{newer3}.

The main result of this paper is the construction of differential codes using Grothendieck's theory of residues for an $r$-dimensional smooth projective variety $X$, with a given set of  points $\mathcal P=\{P_1,\dots,P_n\}$ and {\it any}  effective divisors $\mathcal D=\{ D_1,\dots,D_r\}$ that intersect properly   such that $\mathcal P \subseteq \cap_i D_i$. This construction yields a code that always lies in the dual of the corresponding functional construction, recovers Massman's construction in the case that $D_1,\dots,D_r$ intersect transversally with $\mathcal P = \cap_i D_i,$ and when $r=2$ with $D_1$ and $D_2$ $\Delta$-convenient,  agrees with Couvreur's construction (see \S\ref{sec:couv}). 

This generalization of the differential construction is more than just a technical nicety: given a finite set of points $\mathcal P$ in $X$ there is no need for it to be exactly the transversal intersection of $r$ effective divisors of $X$ (Take for example $\mathcal{P} = \{[1:0:0], [0:1:0], [0:0:1]\}$ in $\mathbb{P}^2$).
 However, by one of  Poonen's  Bertini Theorems \cite[Theorem 3.3]{poonen}, one can find effective divisors $D_1, D_2, \dots, D_r$ such that $\mathcal{P} \subseteq \cap_{i=1}^r D_i$.

Our construction using residues on higher dimensional varieties produces linear codes that behave analogously to differential codes on curves, by satisfying nice properties such as duality (see \S \ref{sec:properties}). The key is to restrict the differential forms in the construction 
so that the analogue of (\ref{linearity}) still holds, by imposing additional vanishing conditions on the forms (such restrictions on differential forms are achieved via the notions of $\Delta$- and sub-$\Delta$-convenient pairs in  \cite{alain2}) . We do this by introducing the idea of a {\it $(\mathcal P,\mathcal D)$-rectifying function} $\theta$, and showing that if a form $\omega$ vanishes at the zeroes of $\theta$, then the analogue of (\ref{linearity}) holds. With this we get a ``rectified"  differential code $C_{\Omega}(\mathcal D, \mathcal{P}, G,\theta)$, which is in the dual 
of $C_L(\mathcal{P}, G).$

We also introduce the class of {\it strictly $(\mathcal P,\mathcal D)$-rectifying} functions $\theta^s$ and the corresponding notion of 
a {\it strictly rectified} differential code $C_{\Omega}(\mathcal D, \mathcal{P}, G,\theta^s)$, which has properties similar to differential codes on curves. In particular, every strictly rectified differential code is a functional code supported on the same set of points, and vice versa. We include a number of examples to illustrate the need for and use of $(\mathcal D,\mathcal P)$-rectifying functions and strictly
$(\mathcal D,\mathcal P)$-rectifying functions.

 We note that estimating the dimension and minimum distance of a rectified differential code will be precisely as difficult as the same thorny problem for functional codes on higher dimensional varieties and we will add nothing about it here.

The paper is organized as follows. In the next section we recall what we need of Grothendieck's theory of residues and use it to construct codes in \S\ref{sec:prelim_construction}.  We present some motivating examples in \S\ref{sec:examples} to illustrate the shortcomings of the previous constructions. Then we present our construction of rectified differential codes in \S\ref{sec:construction}. We derive their main properties in \S\ref{sec:properties} and revisit our previous examples in this new light. 


\section{Grothendieck's Theory of Residues} \label{sec:groth_residues}
The general theory of residues was introduced by Grothendieck \cite{grothendieck} to establish duality theorems for the cohomology of arbitrary proper varieties over perfect fields, although in this paper all the varieties are assumed to be projective over finite fields.  This is studied in detail in \cite{lipman_dualizing} and \cite{hartshorne_residues}. An overview of the theory  that is relevent to this paper can also be found in \cite{lipman_elementary}, \cite{lipman_lecture} and \cite{hopkins_residue},  which we  briefly recall  here. We will also use results from intersection theory. Some comprehensive references are \cite{fulton} and \cite[\href{https://stacks.math.columbia.edu/tag/0AZ6}{Tag 0AZ6}]{stacks-project}, although for our purposes, \cite[Chapter V, \S1 and Appendix A]{hartshorne} will suffice.\\
\indent Let $X$ be an $r$-dimensional smooth projective variety over $k$ and  $\omega$ a $k$-rational $r$-differential regular in an open neighborhood of  $P \in X(k)$. 
Then given any set of $k$-rational generators $\mathbf{x} = \{x_1, x_2, \dots, x_r\}$  of $m_P$, $\omega$ can locally be expressed as a power series 
\begin{align} \label{eqn:expn}
\omega = \sum_{I \in \mathbb{N}^r} c_I \mathbf{x}^I d\mathbf{x},
\end{align}
where  for $I = \{i_1, i_2, \dots, i_r\} \in \mathbb{N}^r$, $\mathbf{x}^I := x_1^{i_1}x_2^{i_2}\dots x_r^{i_r}$, $d\mathbf{x}:=dx_1 \wedge dx_2 \wedge \dots \wedge dx_r $ and $c_I \in k$. For any set of positive integers $a_1,\dots,a_r$,  the  \emph{(Grothendieck) residue}, denoted $Res_P \begin{bmatrix}\omega \\ x_1^{a_1}, x_2^{a_2}, \dots, x_r^{a_r} \end{bmatrix},$ is defined to be
\begin{align}\label{eqn:res'}
Res_P \begin{bmatrix}\omega \\ x_1^{a_1}, x_2^{a_2}, \dots, x_r^{a_r} \end{bmatrix} = c_{\{a_1-1, a_2-1, \dots a_r-1\}.}
\end{align}
  
Recall that a \textit{regular system of parameters} (also known as local parameters) for $O_P$ is a minimal set of generators of $m_P$ and \textit{a system of parameters} $\mathbf f=\{f_1, f_2, \dots, f_r\}$ are elements of $m_P$ such that for some $a>0,$ $m_P^a\subseteq (f_1,\dots,f_r)$, the ideal in $\mathcal O_P$ generated by $\mathbf f$.   For such an $a$, write
\begin{align*}
    x_i^a = \sum_j r_{ij}f_j, \;\;\;\;\;\; r_{ij} \in \mathcal{O}_P.
\end{align*}
We set
\begin{align*}
    R_P(\mathbf{x}, \mathbf{f}, a) := det[r_{ij}] \in \mathcal{O}_P,
\end{align*}
where $\det$ denotes taking the determinant, and define the \emph{residue of $\omega$ relative to $\{f_1, f_2, \dots, f_r\}$}, denoted  $Res_P \begin{bmatrix} {\omega} \\ {f_1, f_2, \dots, f_r} \end{bmatrix} $, to be 

\begin{align} \label{eqn:basechange}
    Res_P \begin{bmatrix}\omega \\ f_1, f_2, \dots, f_r \end{bmatrix} :=  Res_P \begin{bmatrix} R_P(\mathbf{x}, \mathbf{f}, a) \cdot \omega \\ x_1^a, x_2^a, \dots, x_r^a \end{bmatrix}.
\end{align}
 This definition is independent of $a$ and the choice of the generators $\{x_1, x_2 , \dots, x_r\}$ of $m_P$ (\cite[\S7]{lipman_dualizing} and \cite[\S9]{hartshorne_residues}).  
 Suppose we also have $x_i^a = \sum_j r_{ij}^{'}f_j,~ r_{ij}^{'} \in \mathcal{O}_P$. Then  by taking $R=M= \mathcal{O}_P$ in Lemma 7.2 in \cite{lipman_dualizing} and noting that $\mathcal{O}_P$ has depth $r$,   we get 
 \begin{equation}\label{detdiff}
det[r_{ij}]-det[r_{ij}^{'}]\in (x_1^a,\dots,x_r^a),
 \end{equation}
so the definition is also independent of the choice of $r_{ij}$.  

Suppose we are given a set $D_1, D_2, \dots , D_r$ of $k$-rational effective divisors on $X$ that intersect properly.  Let $\mathcal{D}= \{D_1, D_2, \dots, D_r\}$ be the corresponding ordered set. Let $P$ be a $k$-rational point in 
$\cap_i D_i$. (Throughout, any sum or intersection is over the index set $\{1,\dots,r\}$ unless  specified otherwise.) Let $f_i$ denote $k$-rational local equations for $D_i$ in an open neighborhood of $P$ that contains no other point of $\cap_i D_i$. Note that $\{f_1,f_2, \dots, f_r\}$ is a system of parameters of $\mathcal{O}_P$.

\begin{defn}\label{defn:residue}
With notations above, let  $\omega \in \Omega^r(X)$ be such that $f_1f_2\dots f_r \omega$ is a regular  differential in an open neighborhood  of $P$.  We define the residue of $\omega$ relative to $\mathcal{D}$ to be 

\begin{align}\label{eqn:notinintersection}
    Res_P \begin{bmatrix} {\omega} \\ {D_1, D_2, \dots, D_r} \end{bmatrix}  :=\begin{cases} Res_P \begin{bmatrix} f_1f_2\dots f_r{\omega} \\ {f_1, f_2, \dots, f_r} \end{bmatrix}, \text{~if~} P \in \cap_i D_i\\
    0, \text{~otherwise}
    \end{cases}
\end{align}

\end{defn}

\begin{remark}\label{rmk:welldefined} We note that Definition \ref{defn:residue} is well-defined, i.e, that the formula for $Res_P \begin{bmatrix} {\omega} \\ {D_1, D_2, \dots, D_r} \end{bmatrix}$ is independent of the choice of local equations $f_i$ for $D_i$. 

Indeed, let $\{g_1, g_2, \dots, g_r\}$ denote another set of local equations for $\{D_1, D_2, \dots, D_r\}$. Then $f_i = e_i g_i$ for some $e_i \in \mathcal{O}_P^*$ and hence $\mathbf{f} = E \mathbf{g}$ where   $E \in GL_r(\mathcal{O}_P)$ is the diagonal matrix with diagonal entries $\{e_i\}$. From (\ref{eqn:basechange}), we see that 
\begin{align*}
Res_P \begin{bmatrix} g_1g_2\dots g_r{\omega} \\ {g_1, g_2, \dots, g_r} \end{bmatrix} 
   &=  Res_P \begin{bmatrix} det(E) g_1g_2\dots g_r{\omega} \\ {f_1, f_2, \dots, f_r} \end{bmatrix}\\
   &=  Res_P \begin{bmatrix} f_1f_2\dots f_r{\omega} \\ {f_1, f_2, \dots, f_r} \end{bmatrix}.
\end{align*}
 
\end{remark}

\begin{remark}
A similar argument shows a permutation of the $D_i$'s in Definition \ref{defn:residue} changes the sign of the residue at $P$ by the sign of the permutation.
\end{remark}

\begin{remark}
Note that in the above definitions, the point, the differential and the divisors are all defined over $k$, so that the residue will be defined over $k$, which is crucial for applications to coding theory. However it is clear from the definitions that if $k'$ is any algebraic extension of $k$,  the residue is unchanged if we consider the same variety, point, differential, and divisors as being defined over $k'$.  We will use this fact repeatedly and without further comment. 
In particular, Definition \ref{defn:residue} can be applied with $k$ replaced by $\bar{k}$, so we have an unambiguous definition of residue at any point $P\in X(\bar{k}).$
\end{remark}

\begin{remark}
For a divisor $H$ on $X$,  we write $H = H^+ - H^-$ where $H^+$ and $H^-$ are effective. Suppose that
$Supp(H^+)$  is disjoint from $ \cap_i D_i$ and let $D = \sum_i D_i$. Then for  every $\omega \in \Omega^r(D + H) $, it is straightforward to check that $f_1f_2\dots f_r \omega$ is  regular  at  every $P \in \cap_i D_i$.
\end{remark}
\begin{remark}
Note that when $X$ is a curve ($r=1$) and $D$ is the sum of distinct points, the above definition agrees with the usual notion of the residue of a rational differential form that has  poles of order at most one at each point.
\end{remark}

\begin{remark}\label{rmk:rational}
Assuming all $D_i$ and $\omega$ are $k$-rational doesn't insure that
every $P\in \cap_i D_i$ is $k$-rational, but only that $P$ is  rational over its residue field $K=k(P)$, a finite extension of $k$, and as such its set of conjugates over $k$ form a ``closed point" $\bar{P}$
over $k$. In this case   the  definition of a residue at $\bar{P}$ is given by (page 63-64 in \cite{lipman_dualizing}),
$$ Res_{\bar{P}} \begin{bmatrix} {\omega} \\ {D_1, D_2, \dots, D_r} \end{bmatrix}
=Tr_{K/k}(Res_P \begin{bmatrix} {\omega} \\ {D_1, D_2, \dots, D_r} \end{bmatrix}),$$
where $Tr_{K/k}$ denotes taking the trace from $K$ to $k$. In this paper, by extending the base field if necessary, we will always implicitly assume  that every $P \in \cap_iD_i$ is $k$-rational. So we will never need to employ this more general definition.
\end{remark}

Recall that the classical residue theorem for curves states that the sum of residues of a meromorphic differential form is zero. The higher dimensional analogue is \cite[Proposition 12.2]{lipman_dualizing}:
\begin{thm}[Residue Theorem]\label{thm:residuethm}
Let $X$ be a smooth projective variety defined over $k$.
Suppose that $D_1, D_2, \dots,D_r$ are properly intersecting $k$-rational effective divisors.  
Let $\omega \in \Omega^r(\sum_iD_i +H)$ with $H$ a $k$-rational divisor such that $Supp(H^+)$ is disjoint from $\cap_iD_i$.  Then 
\begin{align*}
    \sum_{P \in X(\overline{k})} Res_{P} \begin{bmatrix} \omega \\ {D_1, D_2, \dots, D_r} \end{bmatrix}  = 0.
\end{align*}

Hence by (\ref{eqn:notinintersection}), we get 
\begin{align*}
    \sum_{P_i \in \cap_i D_i} Res_{P_i} \begin{bmatrix} \omega  \\ {D_1, D_2, \dots, D_r} \end{bmatrix}  = 0.
\end{align*}
\end{thm}

\subsection{Comparison with Couvreur's definition of residues for surfaces}\label{sec:couv}
 For a smooth projective surface $X$ over $k$, a rational 2-differential $\omega$ on $X$, and a curve $C$ through a point $P\in X(\bar{k})$, Couvreur \cite{alain} defines the notion of residue of $\omega$  at $P$ along $C$. To allow the reader to compare our results with those in \cite{alain}, we will explain how to interpret Couvreur's residue in terms of Grothendieck's. In particular, we show that for a smooth projective surface, the notion of residue in \cite{alain} agrees with ours  given in Definition \ref{defn:residue}.

In \cite{alain}, the notion of residue  at a point $P$ along $C$ is given by two cases which we will briefly recall here. If $C$ is smooth at $P$, then there is a local equation $v$ for $C$ at $P$ that extends to a pair $(u,v)$ of regular system of  parameters at $P.$ For any rational 2-differential $\omega$ on $X$, Couvreur shows there is a Laurent series in $v$ with coefficients which are Laurent series in $u,$ such that the expansion at $P$ of $\omega$ is of the form
\begin{align}\label{eqn:couv}
\sum_{i\geq  -l}b_i(u)v^i du\wedge dv
\end{align}
for some non-negative integer $l$. The residue of $\omega$  at $P$ along $C$, denoted $res^2_{C,P}(\omega)$, is then defined to be the (1-dimensional) residue of $b_{-1}(u)du$. For more details and properties of $res^2_{C,P}(\omega)$, see \S3 and \S4 in \cite{alain}.

Suppose $C$ is not smooth at $P$ and 
let $\pi:\tilde{X}\rightarrow X$ be a sequence of monoidal tranformations that resolves the singularity of $C$ at $P$. Let $\tilde{C}$ be the strict transformation of $C$ on $\tilde{X}$. By abuse of notation, we will also denote by $\pi$, the induced map $\tilde{C} \rightarrow C $. In this case, the residue of $\omega$  at $P$ along $C$, is defined to be (\cite[Definition 5.2]{alain})
\begin{align*}
    res^2_{C,P}(\omega) = \sum_{Q\in \pi^{-1}(P)}res^2_{\tilde{C},Q}(\pi^*\omega).
\end{align*}

\begin{prop}\label{prop:comparison}
Suppose $\omega\in \Omega^2(X)$ and that $C$ is a $k$-rational curve on $X$ through a $k$-rational point $P$.  Then
\begin{align*}
res^2_{C,P}(\omega)=Res_{P} \begin{bmatrix} \omega  \\{D,rC} \end{bmatrix}
\end{align*}
for some integer $r$ and effective $k$-rational divisor $D$.
\end{prop}

\begin{proof}
For $l$ as in (\ref{eqn:couv}), pick any $r>l$  and an effective $k$-rational divisor $D$ such that $D$ and $rC$  have no components in common and  $(\omega)+D+rC$ is effective. First we establish the claim in the case that $C$ is smooth at $P$. Let $u$ and $v$ be as above, which we can take to be $k$-rational. Let $f$  be a $k$-rational local equation for $D$ at $P$. Then by Definition \ref{defn:residue},

\begin{align*}
Res_{P} \begin{bmatrix} \omega  \\{D,rC} \end{bmatrix}= Res_{P} \begin{bmatrix} fv^r\omega  \\{f,v^r} \end{bmatrix}
\end{align*}
Let $a\geq r$ be such that  $(u,v)^a\subseteq (f,v^r)$.
Now write
\begin{align} \label{eqn:rij}
    u^a = \alpha f+\beta v^r,~v^a=0\cdot f+v^{a-r}\cdot v^r,
\end{align}
for some $\alpha, \beta \in \mathcal O_P$.
Then $R_P(\{u,v\}, \{f,v^r\}, a) = \alpha v^{a-r}$. So by definition,
\begin{align*}
Res_{P} \begin{bmatrix} f v^r\omega \\ {f,v^r} \end{bmatrix} =Res_{P} \begin{bmatrix}\alpha f v^{a-r}v^r\omega\\ {u^a,v^a}\end{bmatrix} = Res_{P} \begin{bmatrix} (u^a-\beta v^r)v^a\omega\\ {u^a,v^a}\end{bmatrix}
\end{align*}
which is the coefficient of $u^{a-1}v^{a-1} du\wedge dv$ in the expansion of
$(u^a-\beta v^r)v^a\omega$.

With notation as above, 
\begin{align*}
(u^a-\beta v^r)v^a\omega=
(u^a-\beta v^r)v^a  \sum_{i\geq -l}b_i(u)v^i du\wedge dv.
\end{align*}

Note that the coefficient of $u^{a-1}v^{a-1} du \wedge dv$ is  the (1-dimensional) residue of $b_{-1}(u)du$ which is exactly the definition  $res^2_{C,P}(\omega)$.

Now suppose that $C$ is not smooth at $P$, and 
let $\pi:\tilde{X}\rightarrow X$ be as above. 
Note that $\pi$ is a surjective birational map over $k$ that is an isomorphism from $\pi^{-1}(X-P)$ onto $X-P$. 
Now by picking $r$ sufficiently large and   applying the Residue Theorem (Theorem \ref{thm:residuethm}) on both $X$ and $\tilde{X},$
we have 

\begin{align*}
Res_{P} \begin{bmatrix}\omega\\{D,rC}\end{bmatrix} &= -\sum_{P'\in X(\overline{k}), P'\neq P}
Res_{P'} \begin{bmatrix} \omega\\ {D,rC}\end{bmatrix}\\
&= -\sum_{P'\in X(\overline{k}), P'\neq P}Res_{\pi^{-1}(P')} \begin{bmatrix}\pi^*\omega\\{\pi^* D,r\tilde{C}}\end{bmatrix}\\
&=-\sum_{Q'\in \tilde{X}(\overline{k}), Q'\not\in \pi^{-1} (P)}Res_{Q'} \begin{bmatrix} \pi^*\omega\\ {\pi^* D,r\tilde{C}}\end{bmatrix}\\
&=\sum_{Q\in \pi^{-1}(P)}
Res_{Q} \begin{bmatrix} \pi^*\omega\\ {\pi^* D,r\tilde{C}}\end{bmatrix}\\
&=\sum_{Q\in \pi^{-1}(P)}res^2_{\tilde{C},Q}(\pi^*\omega)=res^2_{C,P}(\omega).
\end{align*}
Here we're using for all $P' \neq P$ in $X(\overline{k})$, that
\begin{align*}
Res_{P'} \begin{bmatrix} \omega\\ {D,rC}\end{bmatrix}= Res_{\pi^{-1}(P')} \begin{bmatrix}\pi^*\omega\\{\pi^* D,r\tilde{C}}\end{bmatrix},
\end{align*}
which follows from the fact that  $\pi$ is an isomorphism at every point  away from $\pi^{-1}(P)$ (\cite[Proposition 7.13(b)]{hartshorne}). 
\end{proof}

\section{Codes from residues: Preliminary construction} \label{sec:prelim_construction}
Now that we have defined and studied the properties of residues  on higher dimensional varieties, we will give a preliminary construction of codes 
analogous to those for codes attached to
curves. As before, $X$ is an $r$-dimensional smooth projective variety over $k$. Let $\mathcal{D} = \{D_1, D_2, \dots, D_r\}$ be an (ordered) set of $k$-rational divisors that intersect properly. Set $D=\sum_i D_i$. We assume that the points of intersection $\cap_i D_i$ of these divisors are $k$-rational (see Remark \ref{rmk:rational}).  Then given any $\mathcal{P}\subseteq \cap_i D_i$ and  a $k$-rational divisor $G$ whose support is disjoint from $\mathcal{P}$, analogous to (\ref{eqn:res_curves}), we construct the code $C_{\Omega}(\mathcal D, \mathcal{P}, G)$ as the image in $k^n$
of the map
\begin{align}\label{eqn:naiveconst}
    Res_{(\mathcal{D},\mathcal{P},G)}: \Omega^r(D-G) &\rightarrow k^n  \\
       \omega &\mapsto \big (Res_{P_1} \begin{bmatrix} {\omega} \\ {D_1, D_2, \dots, D_r} \end{bmatrix}, 
       \dots, Res_{P_n} \begin{bmatrix} {\omega} \\ {D_1, D_2, \dots, D_r} \end{bmatrix} \big ) \nonumber
\end{align}

Although this is a well-defined construction of codes using residues, we will see that unlike the case of curves, these codes do not always satisfy nice properties such as being dual to functional code $C_L(\mathcal{P}, G)$. We will demonstrate this with some examples.

\section{Some examples} \label{sec:examples}

For a homogeneous polynomial $f \in k[X_0,\dots,X_n]$, we let $V(f)$ denote the algebraic set in $\mathbb{P}^n$ given by $f=0$. For $q$ a power of a prime, we let $\mathbb{F}_q$ denote the field with $q$ elements.

\begin{example} 
Let $k=\mathbb{F}_4 \simeq \mathbb{F}_2[\alpha]/(\alpha^2 +\alpha +1)$. Consider  the projective plane $\mathbb{P}^2$ over $k$ with homogeneous coordinates $X,Y,$ and $Z$.
Let $D_1 = V(X)$ and $D_2 = V(XY^3 + X^2Z^2 + X^2Y^2 + X^3Z+Y^3Z  + Z^2Y^2+ YZ^3)$. We claim that their intersection is $\mathcal{P} = D_1\cap D_2 = \{ [0:0:1], [0:1:0], [0:1:\alpha], [0:1:\alpha+1]\}$. It is clear that $\mathcal{P} \subseteq D_1 \cap D_2$. For the other inclusion, note that by Bezout's theorem (\cite[Corollary I.7.8]{hartshorne}), the total number of points of intersection of $D_1$ and $D_2$ counting multiplicities is equal to the product of degrees of $D_1$ and $D_2,$ which is four. This also shows that the intersection is transversal at each point. Let  $G = V(Y+Z)$, which is disjoint from all points in $\mathcal P$. 

Now let us construct a differential code using the residue map given in (\ref{eqn:naiveconst}).

Let $x = X/Z$ and $y = Y/Z$ be the coordinates of the affine chart $Z\neq 0$. It is easy to check that the following  differential forms  lie in $\Omega^r(D_1+D_2-G)$:
\begin{align*}
    \omega_1 &= \frac{y+1}{x(xy^3 + x^2 + x^2y^2 + x^3 + y^3 + y^2 + y)} dx \wedge dy, \\
    \omega_2 &= \frac{(y+1)x}{x(xy^3 + x^2 + x^2y^2 + x^3 + y^3 + y^2 + y)} dx \wedge dy,\\
    \omega_3 &= \frac{(y+1)y}{x(xy^3 + x^2 + x^2y^2 + x^3 + y^3 + y^2 + y)} dx \wedge dy.
\end{align*}
Now let $\eta$ be any differential in $\Omega^2(\mathbb{P}^2)$, and $(\eta)=K$ be a canonical divisor. Then for any $f\in k(X)^*$, $(f\eta)\geq -(D_1+D_2)+G$ precisely when
$(f)\geq -K-(D_1+D_2)+G$, so $dim~ \Omega^2(D_1+D_2 - G) =dim~L(K + D_1+D_2 -G)$. Since the Picard group of $\mathbb{P}^2$, $Pic~\mathbb{P}^2  = \mathbb{Z}$ (\cite[Corollary II.6.17]{hartshorne}) and $deg~K = -3$ (\cite[Example II.8.20.1]{hartshorne}), we get
 $\mathcal{O}(K + D_1+D_2 -G)\cong \mathcal{O}(-3+5-1) = \mathcal{O}(1)$,
where $\mathcal O(1)$ is the twisting sheaf of Serre, and $\mathcal O(n)$ is its $n^{th}$-power. So by Example 7.8.3, Chapter II in \cite{hartshorne}, we conclude that
\begin{equation*}
    dim~ \Omega^2(D_1+D_2 - G) 
     = 3,
\end{equation*}

Therefore $\omega_1, \omega_2,$ and $\omega_3$ generate $\Omega^r(D_1+D_2-G)$. Now let us compute the residues of these forms at the points in $\mathcal{P}$. \\

\leftline{\underline{Residues at $P_1 = [0:0:1]$}} Since $P_1$ lies in the affine chart $Z\neq 0$, $m_{P_1}$ is generated by $x$ and $y$. Moreover in this chart $D_1$ and $D_2$ are cut out by $f_1=x$ and  $f_2= y+ y^3 + xy^3 + x^2 + x^2y^2 + x^3 +  y^2 $. Let $\mathbf f=\{f_1,f_2\}.$ Note that 
\begin{align*}
    x &=  1 (f_1) + 0(f_2), \\
    y &= \frac{-(y^3 + x + xy^2 + x^2)}{y^2 + y +1} (f_1)+ \frac{1}{y^2 + y +1}(f_2).
\end{align*}
Therefore 
\begin{align*}
    R_{P_1}(\{x,y\},\mathbf f , 1) = \frac{1}{y^2 + y +1}.
\end{align*}
We can now easily compute the residues:
\begin{align*}
  Res_{P_1} \begin{bmatrix} {\omega_1} \\ {D_1, D_2} \end{bmatrix}  &= Res_{P_1} \begin{bmatrix} f_1f_2\omega_1 \\ {f_1, f_2} \end{bmatrix} \\
  &= Res_{P_1}\begin{bmatrix} (y+1)dx\wedge dy \\ {f_1, f_2} \end{bmatrix}\\
  &=Res_{P_1}\begin{bmatrix} R_{P_1}(\{x,y\},\mathbf f, 1) (y+1)dx\wedge dy \\ {x, y} \end{bmatrix}\\
  &=Res_{P_1}\begin{bmatrix} \frac{1}{y^2 + y +1} (y+1)dx\wedge dy \\ {x, y} \end{bmatrix}\\
  &=1.
\end{align*}
Similarly we get
\begin{align*}
    Res_{P_1} \begin{bmatrix} {\omega_2} \\ {D_1, D_2} \end{bmatrix} &= 0,\\
    Res_{P_1} \begin{bmatrix} {\omega_3} \\ {D_1, D_2} \end{bmatrix} &= 0.
\end{align*}
\leftline{\underline{Residues at $P_2 = [0:1:0]$}}

Note that 
$P_2$ lies in the affine chart $Y \neq 0$. Let $s = X/Y$ and $t = Z/Y$ be the coordinates of this affine chart. Note that $m_{P_2}$ is generated by $s$ and $t$. In this chart, $D_1$ and $D_2$ are given by $g_1=s$ and $g_2 = t^3 + t +s + s^2t^2 + s^2 + s^3t + t^2$ respectively. Let $\mathbf g=\{g_1,g_2\}.$ 
Now
\begin{align*}
  s &= 1 (g_1) + 0( g_2), \\ 
  t &= \frac{-(1+st^2+s+s^2t)}{1+t+t^2} (g_1) + \frac{1}{1+t+t^2}( g_2),
 \\
 R_{P_2}(\{s,t\},\mathbf g, 1) &= \frac{1}{1+t+t^2}.
\end{align*}
Moreover, under the change of coordinates $x = s/t$ and $y = 1/t$, the differential forms in the chart $Y \neq 0$ are given by
\begin{align*}
    \omega_1 &= \frac{(t+1)t}{(s)(t^3 + t +s + s^2t^2 + s^2 + s^3t + t^2)}ds \wedge dt,\\
    \omega_2 &= \frac{s(t+1)}{(s)(t^3 + t +s + s^2t^2 + s^2 + s^3t + t^2)}ds \wedge dt,\\
    \omega_3 &= \frac{t+1}{(s)(t^3 + t +s + s^2t^2 + s^2 + s^3t + t^2)}ds \wedge dt.
\end{align*}
The residues can be easily computed as before:
\begin{align*}
    Res_{P_2} \begin{bmatrix} {\omega_1} \\ {D_1, D_2} \end{bmatrix} &= 0,\\
    Res_{P_2} \begin{bmatrix} {\omega_2} \\ {D_1, D_2} \end{bmatrix} &= 0,\\
    Res_{P_2} \begin{bmatrix} {\omega_3} \\ {D_1, D_2} \end{bmatrix} &= 1.
\end{align*}
\leftline{\underline{Residues at $P_3 = [0:1:\alpha]$}}

This point is also in the chart $Y \neq 0$ with coordinates $s,t$. Note that $m_{P_3}$ is generated by $s$ and $t- \alpha$ and 
\begin{align*}
    s &= 1 (g_1) + 0( g_2), \\ 
  t- \alpha &= \frac{-(1+st^2+s+s^2t)}{t(t- (\alpha +1))} (g_1) + \frac{1}{t(t- (\alpha +1))}( g_2),
 \\
 R_{P_3}(\{s,t-\alpha\},\mathbf g, 1) &= \frac{1}{t(t- (\alpha +1))}.
\end{align*}
The residues are 
\begin{align*}
     Res_{P_3} \begin{bmatrix} {\omega_1} \\ {D_1, D_2} \end{bmatrix} &= \alpha +1,\\
    Res_{P_3} \begin{bmatrix} {\omega_2} \\ {D_1, D_2} \end{bmatrix} &= 0\\
    Res_{P_3} \begin{bmatrix} {\omega_3} \\ {D_1, D_2} \end{bmatrix} &= \alpha.
\end{align*}

\leftline{\underline{Residues at $P_4 = [0:1:\alpha + 1]$}}

This point is also in the chart $Y \neq 0$ with coordinates $s,t$. Note that $m_{P_4}$ is generated by $s$ and $t- (\alpha+1)$ and 
\begin{align*}
    s &= 1 (g_1) + 0( g_2), \\ 
  t- (\alpha+1) &= \frac{-(1+st^2+s+s^2t)}{t(t- \alpha)} (g_1) + \frac{1}{t(t- \alpha)}( g_2),
 \\
 R_{P_4}(\{s,t-(\alpha+1)\},\mathbf g, 1) &= \frac{1}{t(t- \alpha) }
\end{align*}
The residues are 
\begin{align*}
     Res_{P_4} \begin{bmatrix} {\omega_1} \\ {D_1, D_2} \end{bmatrix} &= \alpha,\\
    Res_{P_4} \begin{bmatrix} {\omega_2} \\ {D_1, D_2} \end{bmatrix} &= 0,\\
    Res_{P_4} \begin{bmatrix} {\omega_3} \\ {D_1, D_2} \end{bmatrix} &= \alpha+1.
\end{align*}
Therefore we see that $C_{\Omega}(\mathcal{D}, \mathcal{P}, G)$ is  two dimensional and is spanned by $(1, 0, \alpha +1, \alpha)$
and $(0, 1, \alpha, \alpha +1)$.  
\\

Now we will show that this code is dual (orthogonal) to the  corresponding functional code. Note that since $G=V(Y+Z)$, the corresponding  functional code $C_L(\mathcal{P},G)$ is  generated as a $k$-vector space  by the images of  $X/(Y+Z), Y/(Y+Z)$ and $Z/(Y+Z)$ under the evaluation map at $(P_1, P_2, P_3, P_4),$ that is
\begin{align*}
    Ev_{(\mathcal{P},G)}: L(G) &\rightarrow \mathbb{F}_4^4\\
    X/(Y+Z) &\mapsto (0, 0, 0, 0)\\
    Y/(Y+Z) &\mapsto (0, 1, \alpha, \alpha+1) \\
    Z/(Y+Z) &\mapsto (1, 0, \alpha +1, \alpha),
\end{align*}
and hence is 2-dimensional.
One can easily check that  the differential code $C_\Omega(\mathcal D, \mathcal P,G)$ is the dual code to  $C_L(\mathcal{P}, G)$ i.e, $C_\Omega(\mathcal D, \mathcal P,G)=C_L(\mathcal{P}, G)^{\perp}$.  In fact, in this case the code $C_L(\mathcal{P}, G)$ is self-dual and $C_\Omega(\mathcal D, \mathcal P,G)=C_L(\mathcal{P}, G).$
\end{example}

\begin{example} \label{ex:subset}\normalfont
Consider the setting as in the previous example with same $D_1$, $D_2$ and $G$ but with $\mathcal{P}_0= \{ [0:0:1], [0:1:0], [0:1:\alpha] \} \subsetneq D_1 \cap D_2 = \mathcal{P}$.
In this case, we note that $C_L(\mathcal{P}_0, G)$  and $C_\Omega(\mathcal D, \mathcal{P}_0,G)$ are  obtained respectively by puncturing $C_L(\mathcal{P}, G)$ and $C_\Omega(\mathcal D, \mathcal P,G)$ on the last coordinate. Clearly,  $C_\Omega(\mathcal D, \mathcal{P}_0,G) \nsubseteq C_L(\mathcal{P}_0, G)^{\perp}$.  However,  note that $C_L(\mathcal{P}_0, G)^{\perp}$ contains 
\begin{align*}
C_0&=Res_{(\mathcal{D},\mathcal{P}_0, G)}\{\omega\in \Omega^2(D_1+D_2-G)|Res_{P_4}(\omega)=0\}\\
&= Res_{(\mathcal{D},\mathcal{P}_0, G)}((\alpha+1)\omega_1 - \alpha\omega_3, \omega_2)= span((\alpha+1, \alpha, 1)).
\end{align*}
By counting dimensions, we conclude that $C_0  =C_L(\mathcal{P}_0, G)^{\perp}$. \\

From the above example, we see that to have the image of the residue map lying in the dual to  $C_L(\mathcal{P}_0, G)$  one has to restrict the space of differential forms  $\Omega^2(D_1+D_2-G)$ to  the subspace where every $\omega$ satisfies $Res_{P_4}(\omega)=0$.  We will see in the next section  how to formalize this condition in order to obtain a revised definition of the differential code  that will lie in the dual $C_L(\mathcal{P}_0, G)^{\perp}$ .
\end{example}

\begin{example}\label{ex:nontrans}\normalfont
Let $k=\mathbb{F}_9 \simeq \mathbb{F}_3[\alpha]/(\alpha^2 +1)$. Again, let our variety be  $\mathbb{P}^2$ with homogeneous coordinates $X,Y,$ and $Z$. We take $\mathcal{P} = \{ P_1=[1:1:1], P_2 =[-1:1:1], P_3 =[\alpha : 0 :1], P_4 = [- \alpha :0:1], P_5=[0:1:0]\}$ and $G = V(Y+Z)$.
If we let $D_1 = V(Y(YZ - X^2))$ and $D_2 = V(YZ + X^2 - 2Z^2)$. Then $\mathcal{P} \subseteq D_1 \cap D_2$.  We will see below that the intersection is not transversal at $P_5$. Hence by Bezout's theorem we conclude that the intersection multiplicity at $P_5$ is two, and that 
$\mathcal{P} = D_1 \cap D_2$. Note that $P_1, P_2, P_3$ and $P_4$ are contained in the affine chart $Z\neq 0$ with coordinates $x=X/Z, y=Y/Z$ as before. The local equations for $D_1$, $D_2$, and $G$ are given by $f_1=y(y-x^2)$, $f_2 =y+x^2-2$, and $y+1$.
 As before  $dim~ \Omega^2(D_1+D_2 - G) =3$ and is generated by the following differential forms 
 on the affine chart $Z \neq 0,$
\begin{align*}
    \omega_1 &= \frac{(y+1)}{y(y-x^2)(y+x^2-2)} dx \wedge dy,\\
    \omega_2 &= \frac{(y+1)x}{y(y-x^2)(y+x^2-2)} dx \wedge dy, \\
    \omega_3 &= \frac{(y+1)y}{y(y-x^2)(y+x^2-2)} dx \wedge dy.
\end{align*}
 Let $\mathbf f=\{f_1,f_2\}$.  The local parameters at $P_1$, $P_2$, $P_3$ and $P_4$ can be expressed in terms of $\mathbf{f}$ as follows:\\

\underline{At $P_1$:}
\begin{align*}
    x-1 &= \frac{-1}{2y(x+1)} f_1 + \frac{1}{2(x+1)} f_2,\\
    y-1 &= \frac{1}{2y}f_1 + \frac{1}{2}f_2.
\end{align*}
\underline{At $P_2$:}
\begin{align*}
    x+1 &= \frac{-1}{2y(x-1)} f_1 + \frac{1}{2(x-1)} f_2,\\
    y-1 &= \frac{1}{2y}f_1 + \frac{1}{2}f_2.
\end{align*}
\underline{At $P_3$:}
\begin{align*}
    x-\alpha &= \frac{-1}{(y-x^2)(x+\alpha)} f_1 + \frac{1}{(x+\alpha)} f_2,\\
    y &= \frac{1}{(y-x^2)}f_1. 
\end{align*}
\underline{At $P_4$:}
\begin{align*}
x+\alpha &= \frac{-1}{(y-x^2)(x-\alpha)} f_1 + \frac{1}{(x-\alpha)} f_2,\\
    y &= \frac{1}{(y-x^2)}f_1. 
\end{align*}
Computing $R_{P_i}(\mathbf{x}, \mathbf{f}, a)$ for local parameters $\mathbf x$ at these points as in example \ref{ex:subset} yields:
\begin{align*}
    R_{P_1}(\{x-1,y-1\}, \mathbf{f}, 1) &= \frac{1}{y(x+1)},\\
    R_{P_2}(\{x-2,y-1\}, \mathbf{f}, 1) &= \frac{1}{y(x-1)},\\
    R_{P_3}(\{x-\alpha,y\}, \mathbf{f}, 1) &= \frac{-1}{(y-x^2)(x+\alpha)},\\
    R_{P_4}(\{x+\alpha,y\}, \mathbf{f}, 1) &= \frac{-1}{(y-x^2)(x-\alpha)}.
\end{align*}

Note that the point $P_5$ lies in the affine chart $Y\neq 0$ with coordinates $s=X/Y, t=Z/Y,$ as before.  In this chart $D_1$, $D_2$, and $G$ are cut out by $g_1=t-s^2$, $g_2=t+s^2 - 2t^2$, and $1+t$. 
After changing coordinates, the differential forms in this chart  are given by
\begin{align*}
    \omega_1 &= \frac{-t(1+t)}{(t-s^2)(t+s^2-2t^2)} ds \wedge dt,\\
    \omega_2 &= \frac{-s(1+t)}{(t-s^2)(t+s^2-2t^2)} ds \wedge dt,\\
    \omega_3 &= \frac{-(1+t)}{(t-s^2)(t+s^2-2t^2)} ds \wedge dt.
\end{align*}
We note that $\mathbf g=\{g_1, g_2\}$ does not generate $m_{P_5}=(s,t),$ showing that the intersection of the divisors is not transversal at $P_5$. Now
\begin{align*}
    s^2 &= \frac{2t-1}{2(1-t)} (g_1) + \frac{1}{2(1-t)}( g_2), \\ 
  t^2 &= \frac{t}{2(1-t)}(g_1) + \frac{t}{2(1-t)}(g_2),\\
 R_{P_5}(\{s,t\}, \mathbf{g}, 2) &= \frac{-t}{t-1 }.
 \end{align*}
 
 The code $C_\Omega(\mathcal D, \mathcal P,G)$ is the span of the image of the residue map  (\ref{eqn:naiveconst})  evaluated at $\mathcal P=\{P_1,\dots,P_5\}$ and is given by
 \begin{align*}
 Res_{(\mathcal{D},\mathcal{P},G)}: \Omega^2(D-G) &\rightarrow \mathbb{F}_9^5  \\
       \omega_1 &\mapsto (1,2, 2\alpha, \alpha, 0),\\
       \omega_2 &\mapsto (1,1,1,1,2),\\
       \omega_3 &\mapsto (1,2,0,0,0).
\end{align*}
(Note that the residue of $\omega_i$ at $P_5$ is the coefficient of $st$ in the expansion of $R_{P_5}(\{s,t\},\mathbf g,2)\omega_i$.)
On the other hand, the functional code $C_L(\mathcal{P}, G)$ is constructed as in example \ref{ex:subset} and is given by 
\begin{align*}
     C_L(\mathcal{P}, G) = span\{ (2,1, \alpha, 2\alpha, 0), (2,2,0,0,1), (2,2,1,1,0) \}.
\end{align*} 
Note that $C_\Omega(\mathcal D, \mathcal P,G)$ is not contained  in the dual of $C_L(\mathcal{P},G)$. Specifically, the images of $\omega_1$ and $\omega_2$ are contained in the dual of $C_L(\mathcal{P},G)$ but the image of $\omega_3$ is not. The reason for this is that because the intersection of $D_1$ and $D_2$ is not transversal at $P_5$, there are functions $f\in \mathcal L(D_1+D_2-G)$ such that
$$Res_{P_5} \begin{bmatrix} {f\omega_3} \\ {D_1, D_2} \end{bmatrix} \neq f(P_5)Res_{P_5} \begin{bmatrix} {\omega_3} \\ {D_1, D_2} \end{bmatrix}.$$  

We will see in the next section how one can get around this problem to get a differential code that lies in the dual of the functional code. 

\end{example}

\section{A new construction of higher dimensional differential codes} \label{sec:construction}

The examples of the last section show that in general, the  differential construction in \S \ref{sec:prelim_construction} does not yield a code contained in the dual of the corresponding functional code. This happens because the space of differential forms $\Omega^r(\sum_i D_i -G)$ for which we compute the residues is too big in the case that the $D_i$ do not intersect transversally at every point of $\cap_i D_i$ and/or when $\mathcal{P}$ is a proper subset of $\cap_iD_i$. 
We will now see how to fix that problem by restricting ourselves to differentials which vanish on a specified additional divisor, yielding a differential code that 
always lies in the dual of the corresponding functional code.

   As in Section \ref{sec:intro}, let $X$ be an $r$-dimensional smooth projective variety over $k$, $\mathcal P=\{P_1,\dots,P_n\}$ be an ordered finite set of $k$-rational points of $X$, $G$ a $k$-rational divisor of $X$ whose support is disjoint from  $\mathcal P$, $\mathcal{D} = \{D_l\},$  a set of $r$ properly intersecting effective $k$-rational divisors on $X$ containing $\mathcal{P}$, and let $D=\sum_l D_l$. Let $f_{ij}$ be local equations for $D_j$ at $P_i\in \cap_lD_l$, and $\mathbf f_i=\{f_{i1},\dots,f_{ir}\}$, which is a system of parameters at $P_i$. Fix  generators $\mathbf x_i = \{x_{ij}\}$ of $m_{P_i}$ for each $P_i$ and $a_i \in \mathbb{N}$ such that $x_{ij}^{a_i} \in (f_{i1}, f_{i2}, \dots, f_{ir})$ for every $j$.
 
 \begin{defn} \label{defn:rectifying}
 Let $\theta \in k(X)^*$ be regular at each point in $\cap_lD_l$, and for $1\leq i\leq n,$ let $\theta_i$ be its local power series expansion in $\mathbf x_i$ at $P_i$. We say that $\theta$ is $(\mathcal{D}, \mathcal{P})$-rectifying if it satisfies the following conditions locally at each point in $\cap_l D_l$:
   \begin{enumerate}
       \item  $R_{P_i}(\mathbf{x_i}, \mathbf{f_i}, a_i) \cdot \theta_i \equiv c_ix_{i1}^{a_i-1}\dots x_{ir}^{a_i-1} \pmod{x_{i1}^{a_i},\dots,x_{ir}^{a_i}}$, for some $c_i \in k,$ if $P_i \in \mathcal{P}.$
       
       \item  $R_{P_i}(\mathbf{x_i}, \mathbf{f_i}, a_i) \cdot \theta_i \equiv 0 \pmod{x_{i1}^{a_i},\dots,x_{ir}^{a_i}}$ if  $P_i \in \cap_l D_l - \mathcal{P}. $
   \end{enumerate}
  In the above definition, if $c_i \neq 0$ for every $P_i \in \mathcal{P}$, we say that $\theta$ is  strictly $(\mathcal{D}, \mathcal{P})$-rectifying.
   \end{defn}
   
   
   \begin{prop}
   The above definition is independent of the choices of $\mathbf{x_i}$, $\mathbf{f_i},$ and $a_i$ at $P_i$.
   \end{prop}
   \begin{proof}
For  $P \in \cap_lD_l$, let $(\mathbf{x}, \mathbf{f}, a)$ denote the triple corresponding to a choice of  local parameters $\mathbf{x}$, local equations  $\mathbf{f}$ for divisors in $\mathcal{D}$ and $a \in \mathbb{N}$ such that $x_i^{a} \in (f_1, f_2, \dots,  f_r)$.   Let $(\mathbf{y}, \mathbf{g}, b)$ be another triple corresponding to a different choice. It suffices to show that: \\
\begin{align*}
    R_P(\mathbf{x}, \mathbf{f}, a) \theta &= cx_1^{a-1}x_2^{a-1}\dots x_r^{a-1} ~ mod (x_1^a, x_2^a , \dots, x_r^a),\text{for some} ~c\in k, \\ \Leftrightarrow R_P(\mathbf{y}, \mathbf{g}, b) \theta &= dy_1^{b-1}y_2^{b-1}\dots y_r^{b-1} ~ mod (y_1^b, y_2^b, \dots, y_r^b),
    \text{for some} ~d\in k,
\end{align*}
 where $c=0$ if and only if $d=0$.\\

To show this, note that by (\ref{eqn:res'}) and (\ref{eqn:basechange}), for $n_i\geq 0$, the coefficient of $x_1^{a-1-n_1}x_2^{a-1-n_2} \dots x_r^{a-1-n_r}$ in the local expansion of  $R_P(\mathbf{x}, \mathbf{f}, a) \theta$ is equal to 
\begin{align*}
  Res_P \begin{bmatrix} R_P(\mathbf{x}, \mathbf{f}, a) \theta x_1^{n_1} x_2^{n_2} \dots x_r^{n_r} d\mathbf{x}\\ x_1^a, x_2^a, \dots, x_r^a \end{bmatrix} = Res_P \begin{bmatrix} \theta x_1^{n_1} x_2^{n_2} \dots x_r^{n_r} d\mathbf{x} \\ f_1, f_2, \dots, f_r \end{bmatrix}. 
\end{align*}

Hence the condition
\begin{align} \label{eqn:hypothesis}
    R_P(\mathbf{x}, \mathbf{f}, a) \theta = cx_1^{a-1}x_2^{a-1}\dots x_r^{a-1} ~ mod (x_1^a, x_2^a , \dots, x_r^a),
\end{align}
is equivalent to:
\begin{align*}
Res_P \begin{bmatrix} \theta x_1^{n_1} x_2^{n_2} \dots x_r^{n_r} d\mathbf{x}\\ f_1, f_2, \dots, f_r \end{bmatrix}   = \begin{cases}
 0 \text{~~if~~} n_i>0 \text{~~for some}~ i \\
c \text{~~otherwise,}
\end{cases}
\end{align*}
which  is again equivalent to
\begin{align*}
Res_P \begin{bmatrix} \theta z d\mathbf{x}\\ f_1, f_2, \dots, f_r \end{bmatrix}   = cz(P),
\end{align*}
for any $z\in \mathcal O_P$.

By symmetry, to prove the claim, it suffices to show for any $w\in \mathcal O_P$ that  
\begin{align} \label{eqn:equiv}
Res_P \begin{bmatrix} \theta w d\mathbf{y}\\ g_1, g_2, \dots, g_r \end{bmatrix}   = dw(P),
\end{align}
for some $d\in k$, and that $d =0 $ if and only if $c=0$. Let $\mathbf{y} = T \mathbf{x}$ and $\mathbf{f} = E \mathbf{g}$ for some $T, E \in GL_r( \mathcal{O}_P)$, so $\det{T}, \det{E}\in \mathcal{O}_P^*$.   Now by Remark \ref{rmk:welldefined}, we see that
\begin{align*}
    Res_P \begin{bmatrix} \theta w  d\mathbf{y}\\ g_1, g_2, \dots, g_r \end{bmatrix} &= Res_P \begin{bmatrix} \det{E} \cdot \theta w d\mathbf{y}\\ f_1, f_2, \dots, f_r \end{bmatrix} \\
    &= Res_P \begin{bmatrix} \det{E} \det{T} \cdot  \theta w  d\mathbf{x}\\ f_1, f_2, \dots, f_r \end{bmatrix}\\
    &=c \det{E}(P) \det{T}(P) w(P),
\end{align*}
which gives (\ref{eqn:equiv}) with
$d=c \det{E}(P) \det{T}(P)\in k$. Since $\det{T}, \det{E}\in \mathcal{O}_P^*$, $c=0$ if and only if $d=0$.
  \end{proof}
  
Note that in  the above proof we have also shown that:

\begin{cor} \label{cor:leadterm}
Suppose that $\theta$ is $(\mathcal D,\mathcal P)$ rectifying and that $P\in \cap_i D_i.$ Let $D_i$ be locally defined by $f_i$ at $P$, and set $\mathbf f=\{f_1,\dots,f_r\}$. Let $\mathbf x=\{x_i\}$ be local parameters at $P$, and suppose that $x_i^a\in(f_1,\dots,f_r)$ for all $i$.   Then for any $z\in \mathcal O_P$,
\begin{align*}
Res_P \begin{bmatrix} \theta z d\mathbf{x}\\ f_1, f_2, \dots, f_r \end{bmatrix}   = cz(P),
\end{align*}
where $c$ is the coefficient of $x_1^{a-1}x_2^{a-1}\dots x_r^{a-1}$ in the local expansion of $R_P(\mathbf{x}, \mathbf{f}, a) \theta$ at $P$.\\
Taking $z=1$ we get
\begin{align*}
c= Res_P \begin{bmatrix} \theta  d\mathbf{x}\\ f_1, f_2, \dots, f_r \end{bmatrix}.
\end{align*}
\end{cor}

Given a differential $\omega \in \Omega^r(X)$ and a point $P \in X$ with local parameters $\mathbf{x} = \{x_1, x_2, \dots, x_r\}$, let $\omega/d\mathbf{x}$ denote the rational function $w$ where $\omega = w d\mathbf{x}$ on an open neighborhood of $P$.
   
 The usefulness of  $(\mathcal D,\mathcal P)$-rectifying functions is the following:
 
\begin{prop}\label{prop:linearity}
Suppose that $\theta$ is $(\mathcal{D}, \mathcal{P})$-rectifying, and that $h\in k(X)$ is regular at every $P \in \cap_i D_i$. 
Then for any $\omega\in \Omega^r(X)$ such that $f_1f_2 \dots f_r\omega/\theta d \mathbf{x}$ is regular at  every $P\in \cap_i D_i$ (in particular when $\omega\in \Omega^r(D-G-(\theta)^+)$),
\begin{align*}
Res_{P} \begin{bmatrix} h \omega \\ {D_1, D_2, \dots, D_r} \end{bmatrix}&= h(P) Res_{P} \begin{bmatrix} \omega\\ {D_1, D_2, \dots, D_r} \end{bmatrix}.
\end{align*} 
Moreover for $P\notin \mathcal P$,
\begin{align*}
Res_{P} \begin{bmatrix} \omega \\ {D_1, D_2, \dots, D_r} \end{bmatrix}&= 0.
\end{align*}

\end{prop}

\begin{proof}
Take $P\in \cap_iD_i$. By assumption, if $u$ is the function such that
$u d \mathbf{x}=f_1f_2 \dots f_r\omega/\theta,$ then $u$
is regular at $P$. Now by Corollary \ref{cor:leadterm},
\begin{align*}
Res_{P} \begin{bmatrix} h \omega \\ {D_1, D_2, \dots, D_r} \end{bmatrix}&= Res_{P} \begin{bmatrix} h f_1\dots f_r \omega \\ {f_1, f_2, \dots, f_r} \end{bmatrix}=Res_{P} \begin{bmatrix} h u \theta  d\mathbf{x} \\ {f_1, f_2, \dots, f_r}\end{bmatrix}\\
&= h(P)u(P)Res_{P} \begin{bmatrix} \theta  d\mathbf{x} \\ {f_1, f_2, \dots, f_r}\end{bmatrix}\\
&=h(P)Res_{P} \begin{bmatrix} \omega \\ {D_1, D_2, \dots, D_r} \end{bmatrix}.
\end{align*} 
 The proof of second assertion is similar and follows from  Corollary \ref{cor:leadterm} and the fact that for $P \in \cap_iD_i - \mathcal{P}$,  $$c= Res_{P} \begin{bmatrix} \theta  d\mathbf{x} \\ {f_1, f_2, \dots, f_r}\end{bmatrix} =0.$$
\end{proof}

We will need to see how our residues depend on the choices we have made.

\begin{remark}\label{rm:changeofeverything}
With notation as above,  for any $\omega \in \Omega^r(D-G-(\theta)^+)$  and  $u d \mathbf{x} = f_1f_2 \dots f_r\omega/\theta$  where $u$ is regular at  every $P\in \cap_i D_i$,  we have by Corollary \ref{cor:leadterm},
\begin{align}\label{rm:niceresformula}
Res_{P} \begin{bmatrix}  \omega \\ {D_1, D_2, \dots, D_r}
\end{bmatrix}=Res_{P} \begin{bmatrix} u \theta  d\mathbf{x} \\ {f_1, f_2, \dots, f_r}\end{bmatrix}=cu(P),
\end{align}
where $c=Res_{P} \begin{bmatrix}  \theta  d\mathbf{x} \\ {f_1, f_2, \dots, f_r}\end{bmatrix}.$

Now suppose that $G'$ is another $k$-rational divisor of $X$ whose support is disjoint from  $\mathcal P$, $\mathcal{D'} = \{D'_l\}$ is another set of $r$ properly intersecting effective $k$-rational divisors on $X$ containing $\mathcal{P}$ with local equations $f_i'$, that $\theta'$ is $(\mathcal D',\mathcal P)$-rectifying, and that $D-G-(\theta)^+\sim D'-G'-(\theta')^+,$ where $\sim$ denotes linear equivalence. Let $g\in k(X)^*$ be such that
$(g)=(D-G-(\theta)^+)-(D'-G'-(\theta')^+).$ 
Take $\omega\in \Omega^r(D-G-(\theta)^+)$, so
$g\omega\in \Omega^r(D'-G'-(\theta')^+)$.
Then if $u'=f_1'\dots f_r'g\omega/\theta' d{\mathbf x}$,
\begin{align*}
Res_{P} \begin{bmatrix}  g\omega \\ {D'_1, D'_2, \dots, D'_r}
\end{bmatrix}=c'u'(P),
\end{align*}
where $c'=Res_{P} \begin{bmatrix}  \theta'  d\mathbf{x} \\ {f'_1, f'_2, \dots, f'_r}\end{bmatrix}.$
Note that by assumption $gf'_1\dots f'_r\theta/f_1\dots f_r\theta'=u'/u$ 
is regular and non-vanishing at $P,$ and letting $\alpha_P$ be its value at $P,$ we have $u'(P)=\alpha_P u(P)$. Hence if $P\in\mathcal P,$ and $\theta$ and $\theta'$ are {\em strictly rectifying,} so $cc'\neq 0$, then
\begin{align}\label{eq:scalingfactor}
Res_{P} \begin{bmatrix}  g\omega \\ {D'_1, D'_2, \dots, D'_r}
\end{bmatrix}=\beta_P 
Res_{P} \begin{bmatrix}  \omega \\ {D_1, D_2, \dots, D_r}
\end{bmatrix},
\end{align}
where $\beta_P:=(c'/c)\alpha_P\neq 0$ is independent of $\omega$, but does depend on $\mathcal D, G, \theta$, and $\mathcal D', G',$ and $\theta'$. We call $\beta_P$ a {\em scaling factor at $P$}.
Note that 
for any $P\notin \mathcal P$, both residues in (\ref{eq:scalingfactor}) vanish  by Proposition \ref{prop:linearity}, so any $\beta_P\neq 0$ is a scaling factor at $P$.
\end{remark}
\subsection{Existence of rectifying functions}
We will show in a series of lemmas that $ (\mathcal{D}, \mathcal{P})$- (strictly) rectifying functions  exist. We will first show this locally at each point of $\mathcal P$ and then explain how to piece these local functions together.
   
\begin{lemma} \label{lem:local} Let $P\in X$, $\mathbf x=\{x_1,\dots,x_r\}$ be a regular system of parameters and $\mathbf f=\{f_1,\dots,f_r\}$ be any system of parameters at $P$  such that for some $a>0$ 
$x_i^a\in (f_1,\dots,f_r) $ for all $i$. Then there exists  $s_P\in \mathcal{O}_P$ such that
$$R_{P}(\mathbf{x}, \mathbf{f}, a)s_P\equiv x_1^{a-1}\dots x_r^{a-1} \pmod{x_1^{a},\dots,x_r^{a}}.$$

\end{lemma} 

\begin{proof}
  Recall that  $R_{P}(\mathbf{x}, \mathbf{f}, a) = \det{[r_{ij}]}$ where $x_i^a=\sum_{j=1}^rr_{ij}f_j$ with $r_{ij}\in \mathcal O_{P}$. 
 Write $f_j=\sum_{l=1}^r s_{jl}x_l$ with $s_{jl}\in\mathcal O_P$.
 Then 
 \begin{align*}
[r_{ij}][s_{jl}]\begin{pmatrix}x_1\\ \vdots \\ x_r\end{pmatrix} =\Delta \begin{pmatrix}x_1\\ \vdots \\ x_r\end{pmatrix} =\begin{pmatrix}x_1^a\\ \vdots \\ x_r^a\end{pmatrix},
 \end{align*}
 where $\Delta$ is the diagonal matrix with $x_i^{a-1}$ along the diagonal.  
 Since $\mathbf{x}$ is a system of parameters at $P$, we can apply (\ref{detdiff}) to get 
 $$\det{[r_{ij}]det[s_{jl}]}\equiv\det{\Delta}\pmod{x_1^{a},\dots,x_r^{a}}.$$
Taking $s_p = det[s_{jl}]$ yields the lemma.
 \end{proof}

  
  
  To piece together the above result at all of $\mathcal P$ we need two more lemmas. 
  
  \begin{lemma}\label{lem:affine} There exists an affine open subset $U$ of $X$ defined over $k$ such that
  $U$ contains $\cap_l D_l$ and so that for every $P_i \in \cap_l D_l$, $m_{P_i}$ is generated by some $ \mathbf{x_i} = \{x_{i1}, x_{i2}, \dots, x_{ir}\} \subseteq k[U]$.
  \end{lemma}
  \begin{proof}
   This follows from \cite[Proposition 3.3.36]{liu}.
  \end{proof}
  
\begin{lemma} \label{lem:rect}
  Let $\mathcal{Q} = \{Q_1, Q_2, \dots, Q_m\}$ be a finite set of $k$-rational points in a smooth affine $k$-variety $Spec~A$ of dimension $r$  and let $m_{Q_i}= (x_{i1}, x_{i2}, \dots, x_{ir}) \subseteq A$. Then given  $\theta_{i}\in \mathcal{O}_{Q_i}$ and  $a_i>0$, there exists  $\theta\in A$ such that  
  \begin{align*}
      \theta \equiv \theta_{i}\mod {(x_{i1}^{a_i}, x_{i2}^{a_i}, \dots, x_{ir}^{a_i})}.
  \end{align*}
  \end{lemma}
  \begin{proof} This is essentially the Chinese Remainder Theorem.
  Take $l=\max_i{(r(a_i-1)+1)}$, so that $m_{Q_i}^l \subseteq (x_{i1}^{a_i}, x_{i2}^{a_i}, \dots, x_{ir}^{a_i})$.
  The ideals  $m_{Q_i}^l$ and $\prod_{j\neq i} m_{Q_j}^l$ are comaximal. 
  Then for all $i$, there exists  $\gamma_i \in m_{Q_i}^l$ and $\delta_i \in \prod_{j\neq i} m_{Q_j}^l$ such that $\gamma_i + \delta_i =1$. Now it is easy to check that  
  \begin{align*}
      \theta = \sum_{i=1}^r \delta_{i}\theta_i
  \end{align*}
  satisfies the requirements of the Lemma.
  \end{proof}

With $s_{P_i}$ as in Lemma \ref{lem:local} and $c_i \in k$ as in Definition \ref{defn:rectifying},  set $\theta_i=c_is_{P_i}$ if $P_i\in\mathcal P$ and  $\theta_i$ to be $s_{P_i}$ times any element of $m_{P_i}$ if $P_i\in\cap_l D_l-\mathcal P$. Then taking $U$ as in Lemma \ref{lem:affine}, $A = k[U],$ and $Q_i=P_i$ as in Lemma \ref{lem:rect}, we get 
\begin{cor} \label{cor:theta}
For any set $\mathcal{D} = \{D_1, D_2, \dots, D_r\}$ of properly intersecting divisors in $X$ and any $\mathcal{P} \subseteq \cap_l D_l$, 
$(\mathcal{D}, \mathcal{P})$- rectifying functions  as well as  $(\mathcal{D}, \mathcal{P})$- strictly rectifying functions exist.
\end{cor}

\begin{remark}\label{rm:poonen} To make use of this, we will also need the fact that given any finite set $\mathcal P\subset X(k)$, there exists
a set $\mathcal D=\{D_1,\dots,D_r\}$ of properly intersecting effective divisors over $k$ whose intersection contains $\mathcal P$. This follows via an induction argument on \cite[Theorem 3.3]{poonen}.
\end{remark}

\subsection{Rectified Differential Codes}
With this we can now give our construction of rectified differential codes:

\begin{defn}
(Rectified Differential Codes) Let $X$ be an $r$-dimensional smooth projective variety over  $k$,
$\mathcal P=\{P_1,\dots,P_n\}$ be an ordered finite set of $k$-rational points of $X$, and $G$ a $k$-rational divisor of $X$ whose support is disjoint from  $\mathcal P$. Let 
 $\mathcal{D} = \{D_1, D_2, \dots, D_r\}$ be a set of properly intersecting effective divisors over $k$ containing $\mathcal{P}$ (see Remark \ref{rm:poonen}). Set $D=\sum_{i=1}^rD_i$. Let $\theta$ be a $(\mathcal D,\mathcal P)$-rectifying function.  Consider the residue map
\begin{align}\label{eqn:correctdefn}
    Res_{(\mathcal{D},\mathcal{P}, \theta, G)}: \Omega^r(D-G - (\theta)^+) &\rightarrow k^n  \\
       \omega &\mapsto \big (Res_{P_1} \begin{bmatrix} {\omega} \\ {D_1, D_2, \dots, D_r} \end{bmatrix}, \dots, Res_{P_n} \begin{bmatrix} {\omega} \\ {D_1, D_2, \dots, D_r} \end{bmatrix} \big ). \nonumber
\end{align}
We call the image under the above residue map a rectified differential code associated to $(\mathcal{P}, G)$  and denote it by $C_{\Omega}(\mathcal{D},\mathcal{P}, \theta, G)$. If in the above definition, $\theta$ is strictly $(\mathcal{D}, \mathcal{P})$-rectifying, we call the code a strictly rectified differential code associated to $(\mathcal{P}, G)$. We emphasize that a rectifying function $\theta$ is strictly rectifying by adding a superscript $s$ and denoting it as $\theta^s$.   
\end{defn}

\begin{remark} \label{rmk:transrect}
Note that if $\mathcal{P} = \cap_l D_l$ and the divisors in $\mathcal{D}$ intersect transversally at every point of $\mathcal P$, any non-zero constant function is (strictly) 
$(\mathcal{D}, \mathcal{P})$-rectifying and the above construction of rectified differential codes coincides with the construction in \cite{massman}, \cite{alain}, and (\ref{eqn:naiveconst}). This is always the case when $X$ is a curve, whereby we recover Goppa's construction.
\end{remark}

\begin{remark}\label{rm:equivalence}
Suppose $G'$ is another $k$-rational divisor of $X$ whose support is disjoint from  $\mathcal P$, $\mathcal{D}' = \{D'_1, D'_2, \dots, D'_r\}$ is another set of $r$ properly intersecting effective $k$-rational divisors on $X$ containing $\mathcal{P}$, that $\theta'$ is $(\mathcal D',\mathcal P)$-rectifying, and that $\sum D_l-G-(\theta)^+\sim \sum D'_l-G'-(\theta')^+.$ 
By Remark \ref{rm:changeofeverything}, if $\theta$ and $\theta'$ are strictly rectifying, 
$C_{\Omega}(\mathcal{D'},\mathcal{P}, {\theta^s}', G')$  is
obtained from $C_{\Omega}(\mathcal{D},\mathcal{P}, \theta^s, G)$ by multiplying by an invertible diagonal matrix whose diagonal entries are scaling factors at the points of $\mathcal P$. In particular,  the codes  $C_{\Omega}(\mathcal{D'},\mathcal{P}, \theta', G')$  
and $C_{\Omega}(\mathcal{D},\mathcal{P}, \theta, G)$  are equivalent.
\end{remark}

We will now investigate the properties of rectified differential codes.


\section{Properties of rectified differential codes} \label{sec:properties}
The notation $X,\mathcal{P}, G, \mathcal{D}$, $D$ and $\theta$ is as before. In this section we show 
that rectified differential codes on higher dimensional varieties share some of the properties of differential codes on curves discussed in Section \ref{sec:intro}.
\subsection{Strictly rectified differential codes are functional}
We say that a strictly rectified differential code is functional if it  can be obtained via the  functional construction given by (\ref{eqn:functional}) on the same variety and the same set of points. We begin with the following lemma. 
\begin{lemma}\label{lem:localeqn}
There exists an open subset $U\subset X$  containing $\cap_i D_i$ such that each $D_i$ is defined locally on $U$
by some $f_i \in k[U]$.
\end{lemma}
\begin{proof}
One can  reduce to the case where the $D_i$ are  prime divisors. Choose 
a local equation $f_i'$  over $k$ for $D_i$.
Then $(f_i') = D_i + E_i$ for some $E_i$ whose support does not contain $D_i$. By a moving lemma \cite[Theorem III.1.3.1]{shafa}, one can find a rational functions $g_i \in k(X) $ such that $E_i' =  E_i + (g_i)$ contains no points from $\cap_i D_i$. Then $f_i = g_if_i'$ defines the divisor $D_i$ in the open set $U_i = X- supp(E_i')$. Now take $U = \cap_i U_i$.
\end{proof}

\begin{lemma}\label{lem:eta}
With  notation as above, given any strictly $(\mathcal{D}, \mathcal{P})$-rectifying function $\theta^s$, there is an open set $V\subseteq U$  containing $\mathcal{P}$ and a $k$-rational, $r$-differential $\eta$,
such that
\begin{enumerate}[(a)]
\item The divisor of $\eta$ restricted to $V$ is given by 
$
(\eta)|_{V} = -\sum_l (D_l\cap V)
$
and
\item  For all $P\in\mathcal P$, $Res_{P} \begin{bmatrix}  \theta^s \cdot \eta \\ {D_1, D_2, \dots, D_r} \end{bmatrix} = 1$. 
\end{enumerate}
\end{lemma}

\begin{proof}
For $U$ and $f_i$ as in Lemma \ref{lem:localeqn}, set $\mathbf{f}= \{f_1, f_2, \dots, f_r\}$. Fix generators $ \mathbf{x}_i = \{x_{ij}\}$ of $m_{P_i}$ in $k[U]$ for each $P_i \in \mathcal{P}$ and $a_i \in \mathbb{N}$ such that $x_{ij}^{a_i} \in (f_{1}, f_{2}, \dots, f_{r})$ for every $i$ and $j$.  Let $\eta_0$ be a $k$-rational differential $r$-form with no zeros or poles in an open neighborhood $V' \subseteq U$ of $\mathcal{P}$ (such a form exists by the moving lemma \cite[Theorem III.1.3.1]{shafa}). Then $\eta_0$ is given locally at $P_i$ by $h_i d\mathbf{x_i}$ for some $h_i \in \mathcal{O}_{P_i}$ with $h_i(P_i) \neq 0$.\\
 \indent As in Definition \ref{defn:rectifying}, let $c_i$ be the coefficient of $(x_{i1}x_{i2}\dots x_{ir})^{a_i-1}$  in the power series expansion $R_{P_i}(\mathbf{x_i}, \mathbf{f_i}, a_i) \cdot \theta^s$ in $\mathbf x_i$.  Since $\theta^s$ is strictly $(\mathcal{D}, \mathcal{P})$-rectifying, $c_i \neq 0$. Now by Lemma \ref{lem:affine} and Lemma \ref{lem:rect}, we can find $g \in k(X)$ regular at each $P_i$ such that 
 \begin{align*}
 g(P_i)=c_i^{-1}h_i(P_i)^{-1}
 \end{align*}
Note that $g$ is non-vanishing at each point in $\mathcal{P}$. Let $V = V' \cap (X- supp((g)))$. Now setting
 \begin{align*}
     \eta = \frac{g}{f_1f_2\dots f_r} \eta_0
 \end{align*}
and applying (\ref{rm:niceresformula}) yields the lemma.
\end{proof}

\begin{thm} \label{thm:df}
A strictly rectified differential code $C_{\Omega}(\mathcal{D},\mathcal{P}, \theta^s, G)$ is functional. In particular,
\begin{align*}
C_{\Omega}(\mathcal{D},\mathcal{P}, \theta^s, G) = C_L(\mathcal{P}, D - G + (\eta) - (\theta^s)^{-})
\end{align*}
where $\eta$ is as in Lemma \ref{lem:eta}
\end{thm}
\begin{proof}

Consider the map 
 \begin{align*}
    \phi:  L(D - G + (\eta) - (\theta^s)^{-}) &\rightarrow \Omega^r(D - G - (\theta^s)^+),\\
     f &\rightarrow f\cdot \theta^s \cdot \eta.
 \end{align*}
Note that any $ f \in L(D - G + (\eta) - (\theta^s)^{-})$ is regular at $\mathcal{P}$ by Lemma \ref{lem:eta}(a). Now the  map $\phi$ is an isomorphism of $k$-vector spaces (injectivity is clear; for surjectivity  use the fact that $\Omega^r(X)$ is one dimensional over $k(X)$).
By definition, $C_{\Omega}(\mathcal{D},\mathcal{P}, \theta^s, G )$ is the image of the residue map of (\ref{eqn:correctdefn}). Computing the residue at $P_i \in \mathcal{P}$ using Proposition \ref{prop:linearity} and Lemma \ref{lem:eta}(b) we get
\begin{align*}
    Res_{P_i} \begin{bmatrix} f\cdot \theta^s \cdot \eta \\ {D_1, D_2, \dots, D_r} \end{bmatrix} =
    f(P_i)Res_{P_i} \begin{bmatrix} \theta^s \cdot \eta \\ {D_1, D_2, \dots, D_r} \end{bmatrix}
  = f(P_i).
\end{align*}
Hence we get $C_{\Omega}(\mathcal{D},\mathcal{P},\theta^s, G) = C_L(\mathcal{P}, D - G + (\eta) - (\theta^s)^{-})$.
\end{proof}

\subsection{Functional codes are strictly rectified differential codes}
We show that any functional code on $X$  is a strictly rectified differential code, i.e,  it can be realized as a strictly rectified differential code on $X$ using the same set of points. 
\begin{thm}\label{thm:fd}
A functional code $C_L(\mathcal P, G)$ is
a strictly-rectified differential code on $\mathcal P$. In particular, given an ordered set $\mathcal P=\{P_1,\dots,P_n\}\subseteq X(k)$, and a $k$-rational divisor $G$ whose support is disjoint from $\mathcal P$, there is a set $\mathcal{D} = \{D_1, D_2, \dots, D_r\}$ of properly intersecting divisors over $k$ with $\mathcal{P} \subseteq \cap_i D_i$ and a strictly $(\mathcal{D}, \mathcal{P})$-rectifying function $\theta^s$, such that
\begin{align*}
    C_L (\mathcal{P}, G) = C_{\Omega}(\mathcal{D}, \mathcal{P}, \theta^s, D + (\eta) -G - (\theta^s)^{-}),
\end{align*}
where $D=\sum_i D_i$ and $\eta$ is as in Lemma \ref{lem:eta}. 
\end{thm}
\begin{proof}
Given $\mathcal P,$ Remark \ref{rm:poonen} shows that such a $\mathcal D$ exists.
Now by Corollary \ref{cor:theta} there exists a strictly $(\mathcal{D}, \mathcal{P})$-rectifying  function $\theta^s$. 

The result now follows from Theorem \ref{thm:df} by replacing $G$ with $D + (\eta) -G - (\theta^s)^{-}$.
\end{proof}

\subsection{Behavior with respect to taking products}

We will now see that the differential construction (\ref{eqn:correctdefn}) behaves well under taking products of varieties.  

Let $X$ and $Y$ be smooth projective varieties over $k$ of dimensions $r$ and $s$, and suppose that $\mathcal D=\{D_1,\dots,D_r\}$ and $\mathcal E=\{E_1,\dots,E_s\}$ are sets of divisors on $X$ and $Y$ that respectively intersect properly at ordered sets of points $\mathcal P=\{P_1,\dots,P_n\}$ and
$\mathcal Q=\{Q_1,\dots,Q_m\}$. Set $D=\sum_{i=1}^r D_i$ and $E=\sum_{j=1}^s E_j$. Let $X\times Y$ be the product variety of $X$ and $Y$ with projections $pr_X$ and $pr_Y$ onto each factor. We will denote by $pr^*_X$ and $pr^*_Y$, the corresponding pullback functors. Let $\mathcal D\times Y=\{D_i\times Y|D_i\in\mathcal D\}$ and
$X\times \mathcal E=\{X\times E_j|E_j\in\mathcal E\}.$
Then $\mathcal D\times Y\cup X\times \mathcal E$
intersect properly on $X\times Y$ at $\mathcal P\times \mathcal Q$
(which we think of as $\{P_i\times Q_j\}$ ordered lexicographically.)



Now let $G$ be a divisor on $X$ whose support is disjoint from $\mathcal P$ and $H$ a divisor on $Y$ whose support is disjoint from $\mathcal Q$. 

\begin{thm}\label{thm:prod} Suppose that $\theta$ is $(\mathcal D,\mathcal P)$-rectifying on $X$ and that $\lambda$ is $(\mathcal E,\mathcal Q)$-rectifying on $Y$. Let $\mu=pr_X^*(\theta)pr_Y^*(\lambda)$.

a) Let $\omega \in \Omega^r(D - G - (\theta)^{+})$ and
$\chi \in \Omega^s(E - H - (\lambda)^{+})$. Then
if $\psi=pr_X^*(\omega)\wedge pr_Y^*(\chi),$ we have
$$\psi \in \Omega^{r+s}(D\times Y+X\times E- G\times Y-X\times H - (\mu)^{+}).$$

b) 
$\mu$ is $(\mathcal D\times Y \cup X\times \mathcal E,\mathcal P\times \mathcal Q)$-rectifying on $X\times Y$.
Moreover, if $\theta$ and $\lambda$ are strictly rectifying, then so is $\mu.$

c) For $P\in\mathcal P$, $Q\in\mathcal Q$, we have,
$$Res_{P\times Q} \begin{bmatrix} \psi \\ {D_1\times Y,\dots, D_r\times Y, X\times E_1,\dots,X\times E_s} \end{bmatrix}= Res_{P} \begin{bmatrix} \omega\\ {D_1, \dots, D_r} \end{bmatrix} Res_{Q} \begin{bmatrix} \chi\\ {E_1, \dots, E_s} \end{bmatrix}.$$

d) $C_{\Omega}(\mathcal{D}\times Y \cup X\times\mathcal E,\mathcal P \times\mathcal Q, \mu, G\times Y+X\times H)=C_{\Omega}(\mathcal{D},\mathcal{P}, \theta, G)\otimes_k C_{\Omega}(\mathcal{E},\mathcal{Q}, \lambda, H).$

\end{thm}

\begin{proof} a). Pick any $P\in X$ and $Q\in Y$. Let $\mathbf{t} = \{t_1,\dots,t_r\}$ be local parameters at $P$ and $\mathbf{u} = \{u_1,\dots,u_s\}$ be local parameters at $Q$. Then $\{pr_X^*(t_1),\dots,pr_X^*(t_r), pr_Y^*(u_1),\dots,pr_Y^*(u_s)\}$
is a set of local parameters at $P\times Q.$
Then in some neighborhood $U_P$ of $P$, $\omega$ is represented by $f_P d\mathbf{t}$ for some $f_P\in k(X)$, and likewise, in some neighborhood $V_Q$ of $Q$, $\chi$ is represented by $g_Q d\mathbf{u}$ for some $g_Q\in k(Y).$ 
Note that on $U_P$ and $V_Q$ we have $(\omega) = (f_P)$ and $(\chi) = (g_Q),$ respectively.
Now the claim follows from the fact that 
on $U_P\times V_Q,$ $\psi$ is represented by 
$f_Pg_Qd\mathbf{t} \wedge d\mathbf{u}$, which on $U_P\times V_Q$
has divisor $(f_P)\times V_Q+U_P\times (g_Q),$ and as $P$ and $Q$ vary,  $\{U_P \times V_Q\}$ is a cover of  $X \times Y$. 



b) Let $P$, $Q$, $\mathbf{t}$ and $\mathbf{u}$ be as in (a).
Let $\mathbf{f} = \{f_1, f_2, \dots, f_r\}$  and $\mathbf{g} = \{g_1, g_2, \dots, g_s\}$ denote local equations for $\mathcal{D}$ and $\mathcal{E}$ at $P$ and $Q$ respectively, so $pr_X^*(f_i)$ and $pr_Y^*(g_j)$ are respectively local equations for $D_i\times Y$ and $X\times E_j.$
Let $a>0$ be such that both 
$m_P^a\subseteq (f_1,\dots,f_r),$ 
and $m_Q^a\subseteq (g_1,\dots,g_s)$ 
so we can write
\begin{align*}
    t_i^a &= \sum_k r_{ik}f_k, \;\;\;\;\;\; r_{ik} \in \mathcal{O}_P,\\
    u_j^a &= \sum_{l} s_{jl}g_l, \;\;\;\;\;\; s_{jl} \in \mathcal{O}_Q.
\end{align*}
Then applying $pr_X^*$ to the first set of equations and $pr_Y^*$ to the second set gives
\begin{align}\label{eq:prodmult}
    R_{P\times Q}({pr_X^*(\mathbf t)}\cup {pr_Y^*(\mathbf u)},pr_X^*(\mathbf f)\cup {pr_Y^*(\mathbf g)}, a) &= pr_X^*(det[r_{ik}])pr_Y^*(det[s_{jl}])\\ \notag
    &=pr_X^*(R_{P}(\mathbf{t}, \mathbf{f}, a))pr_Y^*(R_{Q}(\mathbf{u},\mathbf{g}, a)).
\end{align}

Since $\theta$ and $\lambda$ 
are rectifying functions, there are $c_P, c_Q\in k$ such that
\begin{align}\label{eq:exp1}
R_{P}(\mathbf{t}, \mathbf{f}, a) \cdot \theta \equiv c_Pt_{1}^{a-1}\dots t_{r}^{a-1} \pmod{t_{1}^{a},\dots,t_{r}^{a}}, \end{align}
\begin{align}\label{eq:exp2}
R_{Q}(\mathbf{u}, \mathbf{g}, a) \cdot \lambda &\equiv c_Qu_{1}^{a-1}\dots u_{s}^{a-1} \pmod{u_{1}^{a},\dots,u_{s}^{a}}.
\end{align}
Applying $pr_X^*$ and $pr_Y^*$ respectively to the relations in (\ref{eq:exp1}) and (\ref{eq:exp2}), and then multiplying (\ref{eq:prodmult}) by $\mu$
shows that the
 first term in the expansion of  
$$R_{P\times Q}({pr_X^*(\mathbf t)}\cup {pr_Y^*(\mathbf u)},pr_X^*(\mathbf f)\cup {pr_Y^*(\mathbf g)}, a)\cdot \mu$$ is
$$c_{P\times Q} pr^*_X(t_1)^{a-1}\dots pr^*_X(t_r)^{a-1} pr^*_Y(u_1)^{a-1}\dots pr^*_Y(u_s)^{a-1},$$
where
\begin{align}\label{prodofcoeff}
c_{P\times Q}:=c_Pc_Q,
\end{align}
so $\mu$ is $(\mathcal D\times Y\cup X\times \mathcal E, \mathcal P\times \mathcal Q)$-rectifying, 
and if $\theta$ and $\lambda$ are strictly rectifying, so is $\mu$.


c) With the notation as in (a) and (b), 
let 
\begin{align}
v=f_1\dots f_r\omega/\theta d\mathbf{t},~ w=g_1\dots g_s\chi/\lambda d\mathbf{u}.
\end{align}
By (\ref{rm:niceresformula}), we have
$$ Res_{P} \begin{bmatrix} \omega\\ {D_1, \dots, D_r} \end{bmatrix} Res_{Q} \begin{bmatrix} \chi\\ {E_1, \dots, E_s} \end{bmatrix}=c_P v(P)c_Q w(Q)=c_{P\times Q}v(P)w(Q),$$
by (\ref{prodofcoeff}).
On the other hand, using (\ref{rm:niceresformula}) again, we have

$$Res_{P\times Q} \begin{bmatrix} \psi \\ {D_1\times Y,\dots, D_r\times Y, X\times E_1,\dots,X\times E_s} \end{bmatrix}= c_{P\times Q}z(P\times Q),$$
where
$$z=pr_X^*(f_1)\dots pr_X^*(f_r) pr_Y^*(g_1)\dots pr_Y^*(g_s)\psi/\mu d (pr_X^*(\mathbf{t}))\wedge d (pr_Y^*(\mathbf{u})).$$
Using the definition of $\psi$ and that the wedge product is bilinear, we have
$$\psi/d (pr_X^*(\mathbf{t}))\wedge d (pr_Y^*(\mathbf{u}))=pr^*_X(\omega/d\mathbf{t}))pr^*_Y(\chi/d\mathbf{u})),$$
so $z=pr_X^*(v)pr_Y^*(w)$. Hence $z(P\times Q)=v(P)w(Q),$ which gives the result.

Finally (d) follows from (c) by the definition of the Kronecker product of matrices.
\end{proof}

\subsection{Orthogonality}
Now we show that rectified differential codes are orthogonal to their corresponding functional codes with respect to the standard dot product on $k^n$.
\begin{thm}\label{thm:duality}
Differential codes are   contained in the dual of functional codes, i.e.,
if $\theta$ is $(\mathcal D,\mathcal P)$-rectifying, then
\begin{align*}
    C_{\Omega}(\mathcal{D},\mathcal{P}, \theta, G) \subseteq C_L(\mathcal{P}, G)^{\perp}.
\end{align*}
\end{thm}
\begin{proof}
Let $f \in L(G)$ and $\omega \in \Omega^r(\sum_i D_i - G - (\theta)^{+})$.
Then by Proposition \ref{prop:linearity}, for every $P_i \in \cap_iD_i$

\begin{align*}
Res_{P_i} \begin{bmatrix} f \omega \\ {D_1, D_2, \dots, D_r} \end{bmatrix}&= f(P_i) Res_{P_i} \begin{bmatrix} \omega\\ {D_1, D_2, \dots, D_r} \end{bmatrix},
\end{align*} 
and
\begin{align*}
Res_{P_i} \begin{bmatrix} \omega \\ {D_1, D_2, \dots, D_r} \end{bmatrix}&= 0 \;\;\;\; \text{if} \; P_i \in \cap_i D_i - \mathcal{P}.
\end{align*}

Applying the Residue Theorem (Theorem \ref{thm:residuethm}), we get
\begin{align*}
   \sum_{P_i \in \cap_i D_i} Res_{P_i} \begin{bmatrix} f\omega\\ {D_1, D_2, \dots, D_r} \end{bmatrix} =  \sum_{P_i \in \mathcal{P}}f(P_i) Res_{P_i} \begin{bmatrix} \omega\\ {D_1, D_2, \dots, D_r} \end{bmatrix} =0.
\end{align*}
The required orthogonality now follows.
\end{proof}
 
Let us now revisit Examples \ref{ex:subset} and \ref{ex:nontrans} and see how $(\mathcal{D}, \mathcal{P})$-rectifying functions rectify the orthogonality issue that shows up  when the intersection  $\cap_i D_i$ is not transversal or when $\mathcal{P}$ is a proper subset of $\cap_i D_i$.

\begin{example} \normalfont
We will first show how using a $(\mathcal{D}, \mathcal{P})$-rectifying function
fixes the difficulty encountered in Example \ref{ex:subset}.
We use the notation of that example.

Recall that the intersection $D_1 \cap D_2$ is transversal at every point of $\mathcal{P}$, but that $P_4\notin \mathcal P_0$.   Note that $\theta^s = \frac{Z -(\alpha+1)Y}{Y+Z}$ is a strictly $(\mathcal{D}, \mathcal{P}_0)$-rectifying function that doesn't vanish at $P_1$, $P_2,$ and $P_3$ but vanishes at $P_4$. In particular, for any $\omega\in \Omega^2(D_1+D_2 - G - (\theta^s)^{+}) $, 
\begin{align*}
Res_{P_4} \begin{bmatrix} \omega\\ {D_1, D_2, \dots, D_r} \end{bmatrix} =0
\end{align*}
Then one can check that $\Omega^2(D_1+D_2 - G - (\theta^s)^{+}) = span \{(\alpha+1) \omega_1 - \alpha \omega_3\}$. 
Hence the image under the residue map on $\Omega^2(D_1+D_2 - G - (\theta^s)^{+})$  yields the code $C_{\Omega}(\mathcal{D},\mathcal{P}_0, \theta^s, G)$ spanned by $(\alpha+1, \alpha, 1),$ which is contained in (in fact equal to) the dual of $C_L(\mathcal{P}_0, G)$. 
\end{example}

\begin{example} \normalfont
We will now show how using $(\mathcal{D}, \mathcal{P})$-rectifying functions
fixes the difficulty encountered in Example \ref{ex:nontrans}.
We use the notation of that example.

Recall that  the intersection of $D_1$ and $D_2$ at $P_5$ is not transversal.  One can check that since $\theta_1 = \frac{Z}{Y+Z} $ and  $\theta_2^s = \frac{X}{Y+Z}$ do not vanish at $P_1, P_2, P_3,$ and $P_4$ but do vanish at $P_5$, they are $(\mathcal{D}, \mathcal{P})$-rectifying
functions. Furthermore, when multiplied by $R_{P_5}(\{s,t\},\mathbf g,2)={{-t}\over{t-1}},$ the leading terms in the expansions of $\theta_1={t\over{1+t}}$, 
$\theta_2={s\over{1+t}}$ at $P_5$ are respectively
$t^2$ and $st$, so $\theta_1$ is rectifying but not strictly rectifying and $\theta_2^s$ is strictly rectifying. 

One computes that $\Omega^2(D_1+D_2 - G - (\theta_1)^{+}) = span \{ \omega_1 \}$ and   $\Omega^r(D_1+D_2 - G - (\theta_2^s)^{+}) = span \{ \omega_2 \}$. 
The respective  images under the residue map  are spanned by $(1,2, 2\alpha, \alpha, 0)$ and $(1,1,1,1,2)$. Therefore we get $C_L(\mathcal{P}, G)^{\perp} = C_{\Omega}(\mathcal{D},\mathcal{P}, \theta_1, G) + C_{\Omega}(\mathcal{D},\mathcal{P}, \theta^s_2, G) $, and neither code gives the whole dual.
\end{example}

In the last example we see that the dual of $C_L(\mathcal{P}, G)$ is the sum of two differential codes. In the next example, we will see that for $ \mathbb{P}^1 \times \mathbb{P}^1 \times \dots \times \mathbb{P}^1$ ($r$-times, for any $r\geq 2$) and a suitable $\mathcal{P}$ and $G$, one can use  arguments similar to \cite[\S 10.2]{alain} to  show that the dual of $C_L(\mathcal{P}, G)$ is equal to the sum of $r$  rectified differential codes that is not a strictly rectified (or equivalently not a functional) code on $\mathcal{P}$.

\begin{example} \label{ex:sum} \normalfont
Consider the variety  $ \mathbb{P}^1 \times \mathbb{P}^1 \times \dots \times \mathbb{P}^1$ ($r$-times) over $k=\mathbb{F}_q$ with homogeneous coordinates $[X_i:Z_i]$ for each $\mathbb{P}^1$ factor.  Let $U$ be the affine chart $\cap_i \{Z_i \neq 0\}$ with coordinates $x_i = X_i/Z_i$. Note that the Picard group of $\mathbb{P}^1 \times \mathbb{P}^1 \times \dots \times \mathbb{P}^1$ is generated by the classes of $E_i$ where $E_i=V(Z_i)$ (\cite[Exercise II.6.1, Corollary II.6.16]{hartshorne}). For any
$0\leq m_i \leq q-2$,  
set $ G_{\{m_i\}}= \sum_i m_i E_i$  and let $\mathcal{P}$ be all the $k$-rational points in $U$. Then $L( G_{\{m_i\}}) \simeq \otimes_i\mathbb{F}_q[x_i]_{\leq m_i},$ where $\mathbb{F}_q[x_i]_{\leq m_i}$ denotes polynomials in $x_i$ of degree at most $m_i$. Therefore  the functional construction yields the code
\begin{align}\label{eqn:func}
    C_L(\mathcal{P},  G_{\{m_i\}}) = \otimes_i RS_q(m_i +1),
\end{align}
where $RS_q(m_i +1)$ denotes the $(m_i +1)$-dimensional  Reed-Solomon code over $\mathbb{F}_q$. Note that the dual of $ C_L(\mathcal{P}, G)$ is (see e.g., \cite[Lemma D.1]{alain})

\begin{align*}
    C_L(\mathcal{P},  G_{\{m_i\}})^{\perp} = \sum_i \mathbb{F}_q^q \otimes \dots \otimes \underbrace{RS_q(q-m_i -1)}_ {i^{th} position}\otimes \dots \otimes   \mathbb{F}_q^q. 
\end{align*}
\indent Let us now construct  $C_L(\mathcal{P},  G_{\{m_i\}})^{\perp}$ using the differential construction. For any given $i$ between $1$ and $r$, consider the family of properly intersecting divisors $\mathcal{D}_i = \{D_{ij}\}_{j=1}^r,$ where $D_{ij} = (f_{ij})^+$ and
\begin{align*}
    f_{ij} &= \prod_{a \in \mathbb{F}_q} (x_j - a) \;\; \text{if} \; j \neq i,\\
    f_{ii} &= \prod_{a \in \mathbb{F}_q} (\sum_l x_l - a). 
\end{align*}
Note that $\mathcal{P}= \cap_jD_{ij}$. At any $\boldsymbol \alpha=(\alpha_1,\dots,\alpha_r)\in \mathcal P$, the local parameters are ${\bf x}-\boldsymbol \alpha=\{x_l-\alpha_l,$ $1\leq l\leq r\}$. 
Using the fact that the product of the non-zero elements of a finite field is $-1$, we deduce that in the local ring $\mathcal{O}_{\boldsymbol \alpha}$, 

\begin{align*}
    f_{ij} &=-(x_j-\alpha_j) \mod{m_{\boldsymbol \alpha}^2}, \text{~~~when~}j\neq i\\
    f_{ii} &= -\sum_{l}(x_l-\alpha_l) \mod{m_{\boldsymbol \alpha}^2}.
\end{align*}

For any $i$, set ${\bf f_i}=\{f_{ij}|1\leq j\leq r\}$. The above yields

\begin{align*}
R_{\boldsymbol \alpha}
({\bf f_i},{\bf x}-\boldsymbol \alpha,1) = (-1)^r \mod{m_{\boldsymbol \alpha}}.
\end{align*}

Hence we conclude that $R_{\boldsymbol \alpha}
({\bf f_i},{\bf x}-\boldsymbol \alpha,1)$ is invertible in $\mathcal{O}_{\boldsymbol{\alpha}}$, so ${\bf f_i}$  is also a regular set of parameters at $\alpha$ and 
the intersection of the divisors in $\mathcal{D}_i$ is transversal at every point in $\mathcal{P}$.

Therefore by Remark \ref{rmk:transrect}, each $C_{\Omega} (\mathcal{D}_i, \mathcal{P},  G_{\{m_i\}})$ is a (strictly) rectified differential code associated to $(\mathcal{P},  G_{\{m_i\}})$. Let
\begin{align*}
   \eta_i = \frac{(-1)^r d\mathbf{x}}{\prod_j f_{ij}}.
\end{align*}
Note that, $\eta_i$ satisfies the conditions for $\eta$ in Lemma \ref{lem:eta}, with
$\mathcal D_i$ playing the role of $\mathcal D$, $U$ playing the role of $V$, and $\theta^s=1$.
Computations yield  $(d{\mathbf x}) =
-2(\sum_l E_l)$,  $(f_{ij})=D_{ij}-qE_j$ for $j\neq i$ and $(f_{ii})=D_{ii}-q\sum_j E_j$. Therefore,

\begin{align*}
    (\eta_i) &= -2(\sum_j E_j) + 2q\sum_{j\neq i} E_j + qE_i - \sum_j D_{ij}\\
    &= \sum_{j\neq i}(2q -2)E_j + (q-2)E_i - \sum_j D_{ij}.
\end{align*}
Now by Theorem \ref{thm:df},
\begin{align*}
C_{\Omega}(\mathcal{D}_i,\mathcal{P},  G_{\{m_i\}}) &= C_L(\mathcal{P}, \sum_j D_{ij} -  G_{\{m_i\}} + (\eta_i) )\\
&= C_L(\mathcal{P}, \sum_{j\neq i}(2q -2)E_j + (q-2)E_i - \sum_j m_j E_j  )\\
&= C_L(\mathcal{P}, \sum_{j\neq i}(2q -m_j -2)E_j + (q-m_i-2)E_i)
\end{align*}

\begin{align} \label{eqn:example}
    \implies C_{\Omega}(\mathcal{D}_i,\mathcal{P},  G_{\{m_i\}}) &= \mathbb{F}_q^q \otimes \dots \otimes \underbrace{RS_q(q-m_i -1)}_ {i^{th} position}\otimes \dots \otimes   \mathbb{F}_q^q, 
\end{align}
since for $m_j \leq q-2$, $RS_q(2q-m_j -1)=\mathbb F_q^q$ for each $j\neq i$.

Therefore 
\begin{align*}
   C_L(\mathcal{P},  G_{\{m_i\}})^{\perp}  = \sum_{i=1}^r C_{\Omega}(\mathcal{D}_i,\mathcal{P},  G_{\{m_i\}}).
\end{align*}

\begin{prop}
The dual $C_L(\mathcal{P}, G_{\{m_i\}})^{\perp}$ above is not a strictly rectified differential code, or equivalently, it is not a functional code over $\mathcal{P}$.
\begin{proof}
 Since by Theorems \ref{thm:df} and \ref{thm:fd} strictly rectified differential codes are  functional and vice-versa, it suffices to show that $C_L(\mathcal{P}, G_{\{m_i\}})^{\perp}$ is not a functional code.  Any functional code on $\mathbb{P}^1 \times \mathbb{P}^1 \times \dots \times \mathbb{P}^1$ evaluated at $\mathcal{P}$ is  obtained by scaling the coordinates of a code of the form (\ref{eqn:func}) by non-zero scalars. This is because for $X_r =\mathbb{P}^1 \times \mathbb{P}^1 \times \dots \times \mathbb{P}^1 $($r$-times),  $Pic(X_r) = \mathbb{Z}^r$ and for any divisor $G$ on $X_r$, $G =  G_{\{m_i\}} + (g)$ for some ${\{m_i\}}$ and  $g \in k(X)$.  Therefore  $C_L(\mathcal{P}, G)$ (where  $G$ is assumed to be  disjoint from $\mathcal{P}$)  is equivalent to  $C_L(\mathcal{P},  G_{\{m_i\}}),$  obtained by scaling the coordinates by a non-zero scalar via the map $L(G) \rightarrow L( G_{\{m_i\}})$ given by  $f \mapsto fg$. (Note that  the condition $G$ is disjoint from $\mathcal{P}$ ensures that $g$ is regular and non-vanishing at  every $P \in \mathcal{P}$).
 
 The claim now follows from the following lemma, which is a generalization of  \cite[Lemma D.2]{alain}.
\end{proof}
\end{prop}

\begin{lemma}
Set $r>1$ and for $1\leq i\leq r$, suppose $n_i\geq 2$. Let $V_i = \mathbb{F}_q^{n_i},$  with the standard basis $e_{i1},\dots,e_{in_i}$.
Let $0\neq U_i \subsetneq V_i$ be proper subspaces. Then the subspace
\begin{align*} 
    \sum_{i=1}^r V_1 \otimes \dots \otimes \underbrace{U_i}_ {i^{th} position} \otimes \dots \otimes   V_r 
\end{align*}
of $V_1 \otimes V_2 \otimes \dots \otimes V_r$
cannot be obtained by applying an invertible scaling function $\phi$ to the coordinates 
(with respect to the basis $\{e_{1j_1}\otimes\dots\otimes e_{rj_r}|1\leq j_l\leq n_l,1\leq l\leq r\} $)
of a subspace of the form $W_1 \otimes W_2 \otimes \dots \otimes W_r$.

\end{lemma}
\begin{proof}
Suppose there exist such $W_1, W_2, \dots, W_r$ so that 
\begin{align}\label{eqn:tensor}
\sum_{i=1}^r V_1 \otimes \dots \otimes \underbrace{U_i}_ {i^{th} position} \otimes \dots \otimes   V_r = \phi(W_1 \otimes W_2 \otimes \dots \otimes W_r),
\end{align}
where $\phi$ is an invertible scaling function on the coordinates of $V_1 \otimes V_2 \otimes \dots \otimes V_r$ as given in the statement of the Lemma. Note that by the hypothesis on $U_i$,
\begin{align*}
    \sum_{i=1}^r V_1 \otimes \dots \otimes \underbrace{U_i}_ {i^{th} position} \otimes \dots \otimes   V_r \subsetneq V_1 \otimes \dots \otimes   V_r.
\end{align*}
Hence $W_i \subsetneq V_i$ for some $i$. Without loss of generality assume that $i=1$. Then there exists a standard basis vector ${e_{1j}}$
not contained in $W_1$.
Now for some non-zero $\mathbf{p} \in V_2 \otimes V_3\otimes \dots \otimes   V_r $,
$e_{1j} \otimes \mathbf{p}$ is in the left hand side of (\ref{eqn:tensor}).
The equality of (\ref{eqn:tensor}) implies that 

\begin{align*}
    {e_{1j}} \otimes \mathbf{p} &\in \phi(W_1 \otimes W_2 \otimes \dots \otimes W_r).
    \end{align*}
 Hence
    \begin{align*}
    \phi^{-1}( {e_{1j}} \otimes \mathbf{p}) &\in W_1 \otimes W_2 \otimes \dots \otimes W_r.
\end{align*}
Note that  
\begin{align*}
\phi^{-1}( {e_{1j}} \otimes \mathbf{p}) =e_{1j} \otimes \psi^{-1}(\mathbf{p})
\end{align*}
for some invertible scaling function $\psi$ of $V_2 \otimes \dots \otimes V_r$ 
with respect to the basis $\{e_{2j_1}\otimes\dots\otimes e_{rj_r}|1\leq j_l\leq n_l,2\leq l\leq r\} $ of $V_2 \otimes \dots \otimes V_r$.

Hence we get 
\begin{align*}
{e_{1j}} \otimes  \psi^{-1}(\mathbf{p}) \in W_1 \otimes W_2 \otimes \dots \otimes W_r
\end{align*}
but this is impossible since $e_{1j} \notin W_1$. 

\end{proof}
\end{example}

The above  example begs the question of whether the the dual of a functional code on an $r$-dimension variety is always the sum of  at most $r$ rectified differential codes. We note that the analogous question for $r=2$ (and $\mathcal P$ the intersection of transversal divisors) was asked and left open in \cite{alain}. 

\section{Acknowledgements}
We would like to thank the referee for a careful reading of a previous draft and helpful comments that greatly enhanced the quality of this manuscript.
\nocite*{}
\bibliographystyle{alpha}
\bibliography{ref_differential}

\end{document}